\documentclass[final,onefignum,onetabnum]{siamart190516}

\usepackage{braket,amsfonts}
\usepackage{stmaryrd}
\usepackage{array}
\usepackage{amsmath,amssymb}
\usepackage[caption=false]{subfig}
\usepackage{pgfplots}

\newsiamthm{claim}{Claim}
\newsiamremark{remark}{Remark}
\newsiamremark{hypothesis}{Hypothesis}
\crefname{hypothesis}{Hypothesis}{Hypotheses}

\usepackage[noend]{algpseudocode}

\usepackage{graphicx,epstopdf}

\Crefname{ALC@unique}{Line}{Lines}

\usepackage{amsopn}

\DeclareMathOperator{\rank}{rank}
\usepackage{xspace}
\usepackage{bold-extra}
\usepackage[most]{tcolorbox}

\colorlet{texcscolor}{blue!50!black}
\colorlet{texemcolor}{red!70!black}
\colorlet{texpreamble}{red!70!black}
\colorlet{codebackground}{black!25!white!25}
\newcommand{\R}{\mathbb{R}}
\newcommand{\C}{\mathbb{C}}
\newcommand{\E}{\mathbb{E}}
\newcommand{\Prob}{\mathbb{P}}
\newcommand{\lowrank}[2]{\llbracket {#1} \rrbracket_{#2}}
\newcommand{\norm}[1]{\left\lVert#1\right\rVert}
\DeclareMathOperator{\sign}{sgn}

\usepackage[normalem]{ulem}

\newcommand{\ignore}[1]{}

\lstdefinestyle{siamlatex}{%
  style=tcblatex,
  texcsstyle=*\color{texcscolor},
  texcsstyle=[2]\color{texemcolor},
  keywordstyle=[2]\color{texemcolor},
  moretexcs={cref,Cref,maketitle,mathcal,text,headers,email,url},
}

\tcbset{%
  colframe=black!75!white!75,
  coltitle=white,
  colback=codebackground, 
  colbacklower=white, 
  fonttitle=\bfseries,
  arc=0pt,outer arc=0pt,
  top=1pt,bottom=1pt,left=1mm,right=1mm,middle=1mm,boxsep=1mm,
  leftrule=0.3mm,rightrule=0.3mm,toprule=0.3mm,bottomrule=0.3mm,
  listing options={style=siamlatex}
}

\newtcblisting[use counter=example]{example}[2][]{%
  title={Example~\thetcbcounter: #2},#1}

\newtcbinputlisting[use counter=example]{\examplefile}[3][]{%
  title={Example~\thetcbcounter: #2},listing file={#3},#1}

\DeclareTotalTCBox{\code}{ v O{} }
{ 
  fontupper=\ttfamily\color{black},
  nobeforeafter,
  tcbox raise base,
  colback=codebackground,colframe=white,
  top=0pt,bottom=0pt,left=0mm,right=0mm,
  leftrule=0pt,rightrule=0pt,toprule=0mm,bottomrule=0mm,
  boxsep=0.5mm,
  #2}{#1}

\patchcmd\newpage{\vfil}{}{}{}
\begin{tcbverbatimwrite}{tmp_\jobname_header.tex}
\title{Randomized low-rank approximation for symmetric indefinite matrices\thanks{Date: \today \funding{TP was supported by the Heilbronn Institute for Mathematical Research.}}}

\author{Yuji Nakatsukasa\thanks{Mathematical Institute, University of Oxford, Oxford, OX2 6GG, UK, (\email{nakatsukasa@maths.ox.ac.uk}, \email{park@maths.ox.ac.uk}).}
\and Taejun Park\footnotemark[2] }

\headers{Indefinite Nystr\"om approximation
}{Yuji Nakatsukasa and Taejun Park}
\end{tcbverbatimwrite}
\input{tmp_\jobname_header.tex}

\ifpdf
\hypersetup{ pdftitle={Randomized low-rank approximation for symmetric indefinite matrices} }
\fi

\begin{document}
\maketitle

\begin{tcbverbatimwrite}{tmp_\jobname_abstract.tex}
\begin{abstract}
The Nystr\"om method is a popular choice for finding a low-rank approximation to a symmetric positive semi-definite matrix. The method can fail when applied to symmetric indefinite matrices, for which the error can be unboundedly large. 
In this work, we first identify the main challenges in finding a Nystr\"om approximation to symmetric indefinite matrices. 
We then prove the existence of a variant that overcomes the instability, and establish relative-error nuclear norm bounds of the resulting approximation that hold when the singular values decay rapidly. The analysis naturally leads to a practical algorithm, whose robustness is illustrated with experiments.
\end{abstract}


\end{tcbverbatimwrite}
\input{tmp_\jobname_abstract.tex}

\section{Introduction}
\label{sec:intro}
Low-rank structures are ubiquitous in the computational sciences. They appear frequently as matrices having low numerical rank \cite{lowrank_udelltownsend}.  A low-rank approximation to a matrix provides an efficient way to store and process the matrix when the dimension is large. The Nystr\"om method \cite{GM,nystrom,SeegarWilliams} has been a popular choice for finding low-rank approximations to symmetric positive semi-definite (SPSD) matrices, especially in the machine learning community for kernel-based methods. 

Let $A\in \mathbb{R}^{n\times n}$ be a SPSD matrix and let the positive integer $r$ be the target rank. Then the Nystr\"om method is given by $A_{nys}^{(s)} = CW^\dagger C^T$ where $C:=AX\in \R^{n\times s}$ and $W:=X^T\!AX\in \R^{s\times s}$  with $r\leq s<n$ and $X\in \R^{n\times s}$ is a sketching matrix. The positive integer $s$ is called the sketch size, and typically $r < s\ll n$. Traditionally, $X$ is chosen to be a column sampling matrix, which has exactly one non-zero entry equal to $1$ in each column \cite{GM,SeegarWilliams}. In this case, $C$ is a subset of $s$ columns of $A$ and $W$ is an $s\times s$ principal submatrix of $A$. There are different sampling schemes for column sampling, including uniform sampling, leverage score sampling \cite{GM,curleverage,SeegarWilliams,woodruff,muscomusco2017} and k-means++ sampling \cite{oglicGartner2017}. In recent years, other choices for $X$ have been shown to be practical, including Gaussian matrices, subsampled randomized trigonometric transforms (SRTTs) and sparse maps \cite{hmt,martinsson_tropp_2020}. These are \emph{random embeddings}, which are the focus of this paper, and unlike column sampling, they mix up the coordinates of a vector when applied \cite{martinsson_tropp_2020}. 

In this paper, we investigate the effect of using $A_{nys}^{(s)}$ and its rank-restricted variants for symmetric matrices that are possibly \emph{indefinite}. Low-rank approximation of symmetric indefinite matrices arises in many applications, such as learning in reproducing kernel Kre{\u i}n spaces \cite{indef1}, natural language processing \cite{devlinetal2019,PiccoliRossi2014} and non-metric proximity transformations \cite{indef2}, which has applications in bioinformatics and social networks. The original matrix $A$ does not have to be SPSD for one to form the Nystr\"om approximation $A_{nys}^{(s)}$. However, the theory does not translate directly to symmetric indefinite matrices because it uses the fact that the original matrix is SPSD \cite{specerr,GM,WGM}. Indeed, the Nystr\"om approximation can be very poor for indefinite $A$, as we illustrate below. In this work, we show that a judiciously constructed rank-restricted variant of the Nystr\"om approximation, when used with random embeddings, is robust even for symmetric indefinite matrices, which often outperforms other existing methods as we show for synthetic datasets (Figure \ref{indefkerfig}) and real datasets (Figure \ref{fig:UCI}) in Section \ref{sec:numexp}. We also show in Section \ref{sec:thm} that there exists a projection for the core matrix $W$ such that the Nystr\"om approximation gives a good low-rank approximation to \emph{any} symmetric matrix when the singular values decay sufficiently fast.


\subsection{Nystr\"om methods and related work}
There are several variants of the Nystr\"om method for SPSD matrices. There are two rank-restricted versions that give a rank-$r$ approximation to $A_{nys}^{(s)}$ where $r<s$. The first version, which is more traditional, is defined by $A_{nys}^{(s,r)} = C\lowrank{W}{r}^\dagger C^T$ \cite{drineasmahoney05,GM,LiNysRSVD} where $\lowrank{W}{r}$ denotes the best rank-$r$ approximation to the matrix $W$ using the truncated SVD. The second version is given by $\lowrank{A_{nys}^{(s)}}{r} = \lowrank{CW^\dagger C^T}{r}$ \cite{P-AB,TroppRR2,WGM}, which was suggested more recently. The difference between the two methods is that $A_{nys}^{(s,r)}$ performs rank-truncation in the core matrix, $W$, which makes this method cheaper to compute, while $\lowrank{A_{nys}^{(s)}}{r}$ performs rank-truncation in the Nystr\"om approximation $A_{nys}^{(s)}$, which makes this method take advantage of the full Nystr\"om approximation, $C$ and $W$, when performing the rank-truncation. There are also other variants of the Nystr\"om method, including one for rectangular matrices \cite{GN,TroppGN} and one that guarantees numerical stability \cite{GN}. This paper will mostly focus on $A_{nys}^{(s,r)}$.

It is known that for SPSD matrices, $A_{nys}^{(s)}$ \cite{GM} and $\lowrank{A_{nys}^{(s)}}{r}$ \cite{WGM} satisfy relative-error bounds in the nuclear norm. This means that if $\hat{A}$ is a low-rank approximation to $A$ (in this case, $A_{nys}^{(s)}$ or $\lowrank{A_{nys}^{(s)}}{r}$) and $\epsilon>0$ then
\begin{equation} \label{relbound}
    \norm{A-\hat{A}}_* \leq (1+\epsilon)\norm{A-\lowrank{A}{r}}_*
\end{equation}
holds with high probability under some conditions on the sketch $X$ and the sketch size $s >r$ where $\norm{\cdot}_*$ is the nuclear norm (the sum of the singular values). The details are in the relevant papers \cite{GM,WGM}. On the other hand, it is not known whether $A_{nys}^{(s,r)}$ satisfies a relative-error norm bound mentioned above \cite{WGM}. In \cite{P-AB}, an example of a $3\!\times\!3$ SPSD matrix is given, showing the downside of using $A_{nys}^{(s,r)}$ for kernel approximations which commonly uses a column sampling matrix. The authors propose $\lowrank{A_{nys}^{(s)}}{r}$\footnote{As in \cite{P-AB}, for SPSD matrices, it should be noted that $\norm{A-\lowrank{A_{nys}^{(s)}}{r}}\leq \norm{A-A_{nys}^{(s,r)}}$ will hold in the spectral norm and the Frobenius norm.} as an alternative, for which later Wang, Gittens and Mahoney derived a
 relative-error norm bound \cite{WGM}. For this example, the problem persists even if we use random embeddings. However, this is a small example that can yield results with high variability, and random embeddings do give a smaller expected relative-error in the nuclear norm and a smaller variance result than column sampling, especially when the dimension of the matrix is large. This hints that random embeddings can be more robust and reliable than column sampling. This type of phenomena have been discussed before, for example in \cite{martinsson_tropp_2020} where the authors point out that column sampling is less reliable than random embeddings due to their relatively high variance results.

For symmetric indefinite matrices, which are the focus of this paper, not much has been shown. It is however known that the problem is rather difficult. We can easily see that the plain Nystr\"om approximation, $A_{nys}^{(s)}$ can behave poorly for symmetric indefinite matrices. We can easily see that the plain Nystr\"om approximation, $A_{nys}^{(s)}$ can be very bad for symmetric indefinite matrices. For example, let $0<\epsilon<1$ and
\begin{equation}\label{eq:badexample}
    A = \begin{bmatrix}
        0 & 1 \\ 1 & 0
    \end{bmatrix}, X = \begin{bmatrix} \epsilon \\ \sqrt{1-\epsilon^2}
    \end{bmatrix}
\end{equation} where $A$ has eigenvalues $\pm 1$. Then the plain rank-$1$ Nystr\"om approximation to $A$ is
\begin{equation}
A_{nys}^{(1)} = AX(X^TAX)^\dagger X^TA = \frac{1}{2\epsilon \sqrt{1-\epsilon^2}}\begin{bmatrix}
1-\epsilon^2 & \epsilon \sqrt{1-\epsilon^2} \\ \epsilon \sqrt{1-\epsilon^2} & \epsilon^2
\end{bmatrix}
\end{equation} and therefore
\begin{equation}
    \norm{A-A_{nys}^{(1)}}_* = \frac{1}{2\epsilon \sqrt{1-\epsilon^2}},
\end{equation} which can be arbitrarily large as $\epsilon\rightarrow 0$, whereas the best rank-$1$ nuclear norm error of $A$ is $1$. This type of issue has also been observed in a different context for a CUR approximation of rectangular matrices \cite{cortinovis2020low}. Essentially, the issue arises from the presence of an eigenvalue of $X^TAX$ close to (or even equal to) 0, much smaller than $\sigma_{r}(A)$ or even $\sigma_{\min}(A)$---a phenomenon that is absent when $A$ is SPSD. This blows up the norm of the core matrix $(X^TAX)^\dagger$, causing instability. While this is admittedly a contrived example, the difficulty can be easily observed also in experiments. In Figure \ref{nys_fail}, the two plots were generated using $100\times 100$ symmetric indefinite matrices with Haar distributed eigenvectors. In the left plot, the eigenvalues decay geometrically from $1$ to $10^{-8}$ with random signs, and in the right plot, the first $20$ eigenvalues are equal to $\pm 1$ and the other $80$ eigenvalues are equal to $\pm 10^{-10}$ where the signs were applied randomly with equal probability. We apply the plain Nystr\"om approximation $A_{nys}^{(r)}$ using the Gaussian sketch to $A$. We can see that the plain Nystr\"om approximation can be unstable.
\begin{figure}[!ht]
\centering
\includegraphics[scale = 0.6]{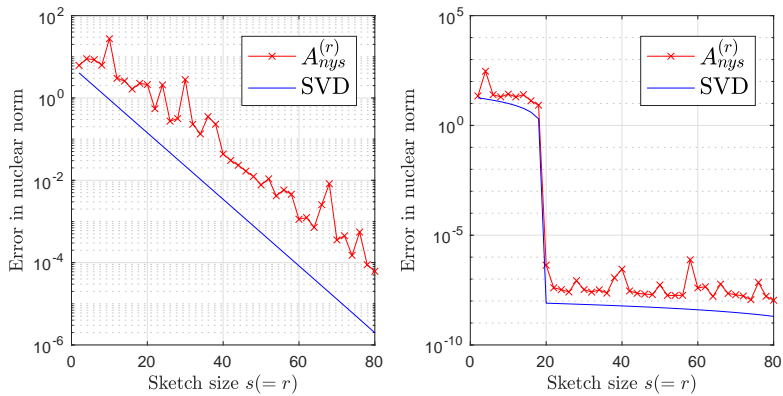}
\centering
\caption{Plain Nystr\"om approximation $A_{nys}^{(r)}$ using the Gaussian sketch to $100\times 100$ symmetric indefinite matrices. We can see that $A_{nys}^{(r)}$ can be unstable.}
\label{nys_fail}
\end{figure}
This type of issue has also been observed in a different context for CUR approximations of rectangular matrices \cite{cortinovis2020low}. Essentially, the issue arises from the possible presence of an eigenvalue of $X^T\!AX$ much smaller than $\sigma_{r}(A)$ or even $\sigma_{\min}(A)$---a phenomenon that is absent when $A$ is SPSD. This blows up the norm of the core matrix $(X^T\!AX)^\dagger$, causing instability.

\paragraph{Contributions}
Our first contribution is to identify the main challenges in finding a good Nystr\"om approximation to symmetric indefinite matrices. We find that the accuracy of the Nystr\"om method is related to controlling the singular values of the core matrix $W = X^T\!AX$, and show that the accuracy can be lost even if the singular values of $W$ are sufficiently larger than the unit roundoff if $W$ severely underestimates the leading eigenvalues of $A$. We then perform an analysis in Section \ref{sec:thm} that overcomes the challenges. The analysis shows that a certain truncation in the core matrix can give a reliable Nystr\"om approximation that guarantees \eqref{relbound} to symmetric indefinite matrices when the singular values decay sufficiently quickly. To our knowledge, this is the first relative-error norm bound for the Nystr\"om method concerning general symmetric matrices that are possibly indefinite.

Our second contribution is providing a practical algorithm (Algorithm \ref{alg:nys}) that gives a Nystr\"om approximation to symmetric indefinite matrices. We show its robustness by comparing the algorithm to some of the existing methods in Section \ref{sec:numexp} and show that the algorithm performs robustly for symmetric indefinite matrices even in the presence of small singular values in the core matrix, whereas the other algorithms can fail. This algorithm is not new in the context of the Nystr\"om method for SPSD matrices. However, to our knowledge, it has not been suggested or studied before for symmetric indefinite matrices.

\paragraph{Existing methods} 
We review three existing ideas for using the Nystr\"om method for indefinite matrices, among others. Cai, Nagy and Xi~\cite{indefnys} derive an error bound for the Nystr\"om method, $A_{nys}^{(s)}$ for symmetric indefinite matrices that arise from a symmetric function. This bound depends on how close the function values of the sampled points are, which is not an attractive dependence and may not be very useful in more general or practical situations. They suggest the plain Nystr\"om method $A_{nys}^{(r)}$, which can be unstable. They also suggest $AX(X^T\!AX)^\dagger_{\epsilon} (AX)^T$ for the Nystr\"om approximation motivated by \cite{GN} with the aim of improving the stability. This version truncates the core matrix $W = X^T\!AX$ so that $\sigma_{\min}((X^T\!AX)_\epsilon)>\epsilon$ where $\epsilon$ is of the order of the unit roundoff. However, this version can give worse approximations than $A_{nys}^{(s)}$ \cite{indefnys} and does not always improve the stability of the Nystr\"om approximation. Second, Ray \emph{et al.}~\cite{smsnystrom} suggest submatrix-shifted (SMS) Nystr\"om to provide an efficient algorithm that deals with symmetric matrices that have only few negative eigenvalues. This method uses an eigenvalue shift based on the minimum eigenvalue of a small principal submatrix before applying the plain Nystr\"om method $A_{nys}^{(r)}$. The downside of this method is that the eigenvalue shift can have serious negative impact on the approximation quality. Lastly, the authors in \cite{indef2,indef1} devise strategies to form the Nystr\"om approximation to symmetric indefinite matrices. However, these methods use eigenvalue information of the original matrix, which is expensive to compute. The three existing methods described above use column sampling matrices for $X$, which is different from random embeddings. In the final section (Section \ref{discussion}), we will revisit their differences in relation to our method and discuss the implications.

\paragraph{Non-Nystr\"om approaches} In \cite{hmt}, a low-rank approximation for symmetric matrices in the form of the randomized SVD is given. This approximation is given by $QQ^TAQQ^T$ where $Q\in \R^{n\times s}$ is the orthonormal matrix in the thin QR decomposition of $AX$ and is known to satisfy a relative-error norm bound. The dominant cost is $O(n^2 s)$ flops for forming $Q^TA$ (assuming $A$ is dense), which becomes prohibitive when $n,s$ are large. Wang, Luo and Zhang derived in \cite{Wang2014} a relative-error norm bound to any symmetric matrices (possibly indefinite) for the prototype model. This model computes the low-rank approximation by first forming the sketch $C = AX$ and then approximating $A$ by $CXC^T$ where $X = C^\dagger A (C^\dagger)^T$. The authors show that if $C$ contains $s = O(k/\epsilon)$ columns of $A$ chosen by adaptive sampling then the prototype model has relative-error of at most $(1+\epsilon)$. The dominant costs for the algorithm in \cite{Wang2014} are $O(n^2 r \log r)$ for computing $C$ and $O(n^2r)$ for computing $C^\dagger A$, which becomes very costly with large $n$.

\paragraph{Non-symmetric approaches} We can use non-symmetric low-rank approximation to symmetric indefinite matrices. Examples are the randomized SVD \cite{hmt}, which is given by $QQ^TA$ using the notation in the previous paragraph and the generalized Nystr\"om method \cite{clarksonwoodruff14,GN,TroppGN} given by $AX(Y^T\!AX)^\dagger Y^TA$ where $X$ and $Y$ are independent random embeddings of different dimensions. The details can be found in the relevant papers. For both methods, since their representation is not symmetric, if we want to force symmetry in their representations (e.g. by taking the symmetric part $(M^T+M)/2$), we may risk doubling the rank in the approximation. In addition, as mentioned in the previous paragraph, the randomized SVD has the cost of computing $Q^TA$, which becomes prohibitive when $n,s$ are large. For generalized Nystr\"om, we approximately double the number of matrix-vector multiplications needed as $A$ needs to be multiplied by two independent random embeddings $X$ and $Y$ and this, in turn doubles the storage requirement (in fact, more than double because $Y$ (or $X$) is recommended to be larger~\cite{GN}). In this paper, we focus on symmetric low-rank approximations.

\paragraph{Notation} Throughout, we use $\norm{\cdot}_2$ for the spectral norm or the vector-$\ell_2$ norm, $\norm{\cdot}_*$ for the nuclear norm (sum of singular values) and $\norm{\cdot}_\text{F}$ for the Frobenius norm. We use dagger $^\dagger$ to denote the pseudoinverse of a matrix and $\lowrank{A}{r}$ to denote the best rank-$r$ approximation to $A$ in any unitarily invariant norm, i.e., the approximation derived from truncated SVD~\cite{hornjohn}. Unless specified otherwise, $\sigma_i(A)$ denotes the $i$th largest singular value of the matrix $A$ and $\lambda_i(A)$ the $i$th largest eigenvalue in magnitude. Lastly, we use MATLAB style notation for matrices and vectors. For example, for the $k$th to $(k+j)$th columns of a matrix $A$ we write $A(:,k:k+j)$.


\section{Proposed method}
\label{sec:alg}
When we use the Nystr\"om method on symmetric indefinite matrices, it can lead to problems. The main concern is in the core matrix $W = X^T\!AX$ because the positive and negative eigenvalues of $A$ can `cancel' each other out when forming $W$, making the eigenvalues of $W$ much smaller than $\sigma_r(A)$. This causes inaccuracies and instabilities when computing the pseudo-inverse of $W$. More specifically, if we use column sampling then $W$ would be a principal submatrix of $A$. By Cauchy's interlacing theorem, the spectrum of $W$ is contained in the interval $[\lambda_{\min}(A),\lambda_{\max}(A)]$ which contains both positive and negative values since $A$ is indefinite. Therefore the magnitude of the eigenvalues of $W$ can be significantly smaller in magnitude from those of $A$, resulting in the matrix $W^\dagger$ blowing up. In addition, the computation of the pseudo-inverse of $W$ can be numerically unstable if $\sigma_{\min} (W)<u$ where $u$ is the unit roundoff. Thus, the main challenge is to ensure that $W^\dagger$ does not ruin the Nystr\"om approximation quality. One approach is to introduce a potentially large shift to make $A$ SPSD, but this can severely affect the approximation quality unless $A$ is nearly definite, that is, the negative eigenvalues of $A$ are very small in magnitude, for example, on the order of machine precision. This idea is used for SPSD matrices where a small shift is introduced to gain numerical stability, however the shift here needs to be small enough to ensure that accuracy is still high \cite{stablenys,TroppRR2}. 

In light of these observations, we propose
\begin{equation*}
    A_{indef}^{(c,r)}  = AX \lowrank{X^T\!AX}{r}^\dagger (AX)^T
\end{equation*} for symmetric \emph{indefinite} matrices $A\in \R^{n\times n}$ where $X\in \R^{n\times cr}$ is a random embedding, $c>1$ is a modest constant, say $c = 1.5$ or $c = 2$, and $r$ is the target rank. When $A$ is SPSD and the sketch size $s$ is proportional to the target rank, $A_{indef}^{(c,r)}$ is equivalent to $A_{nys}^{(cr,r)}$. This rank-restricted version truncates the bottom $(c-1)r$ singular values of $W\in \mathbb{R}^{cr\times cr}$, which can potentially be harmful even if they are sufficiently larger than the unit roundoff. This is different to the truncation used in \cite{indefnys} as they use truncation based on the magnitudes of the singular values of $W$, whereas for our method, the number of bottom singular values we truncate is proportional to the target rank. This intuition is justified by Andoni and Nguy$\hat{e}$n~\cite{andoni}, who prove that the largest eigenvalues (whose proportional to the sketch size) of symmetric matrices with rapidly decaying singular values  are approximately preserved under conjugation by a Gaussian sketch with an appropriate normalization factor. 

Now, let us define a quantity that will measure how well the singular values are preserved in the core matrix $W$ of the Nystr\"om method. For a symmetric matrix $A\in \R^{n\times n}$, a target rank $r$ and a sketch size $s \geq r$, define
\begin{equation}
    \kappa_W (A,r,s):= \frac{\max\limits_{1\leq i \leq r}\sigma_i(X^T\!AX)/\sigma_i(A)}{\min\limits_{1\leq j \leq r}\sigma_j(X^T\!AX)/\sigma_j(A)} = \max\limits_{1\leq i \leq r}\max\limits_{1\leq j \leq r}\frac{\sigma_i(X^T\!AX)}{\sigma_j(X^T\!AX)}\frac{\sigma_j(A)}{\sigma_i(A)}
\end{equation}
where $X\in \R^{n\times s}$ is a Gaussian embedding matrix. This quantity measures the ratio between the worst over-approximation and the worst under-approximation of the leading singular values of $A$ using the singular values in the core matrix $W$. $\kappa_W(A,r,s)$ will help us see how much the singular values of $W$ have deviated from the leading singular values of $A$, which directly affects the Nystr\"om approximation quality as we illustrate below.

In Figure \ref{Wcomp}, we show how important it is to ensure that the spectrum of $W$ does not ruin the approximation quality. In this experiment\footnote{All experiments were performed in MATLAB version 2021a using double precision arithmetic.}, $A\in \R^{1000\times 1000}$ is a symmetric indefinite matrix constructed as in the left plot of Figure \ref{nys_fail}. The smallest singular value in the core matrix was larger than $10^{-7}$ throughout this experiment. For the truncated cases, $A_{indef}^{(1.5,r)}$ and $A_{nys}^{(r+5,r)}$, the approximation is robust as seen in Figure \ref{fig:2a}. This robustness we see is illustrated in Figure \ref{fig:2b} where the singular values of $W=X^T\!AX$ behaves well in the sense that there is no wild fluctuations in $\kappa_W(A,r,r+5)$ and $\kappa_W(A,r,1.5r)$. However, when the sketch size is not proportional to the target rank ($s = r+5$), the relative approximation error for $A_{nys}^{(r+5,r)}$ (when compared with the truncated SVD) and $\kappa_W(A,r,r+5)$ grow as we increase the target rank. This problem can become worse and the approximation can become unstable when we use SRTT matrices for efficiency with the sketch size $s = r+5$ (See Figure \ref{proposedalg} and Subsection \ref{subsec:randemb}). When the sketch size is proportional to the target rank, $\kappa_W(A,r,1.5r)$ and the relative approximation error for $A_{indef}^{(1.5,r)}$ are approximately a constant, which motivates us to choose the oversample size to be proportional to the target rank. On the other hand, without the truncation in the core matrix we see that $\kappa_W(A,r,r)$ behaves wildly. This indicates that the singular values of $W$ inaccurately approximates the leading singular values of $A$. As a result, the Nystr\"om approximations $A_{nys}^{(1.5r)}$ and $\lowrank{A_{nys}^{(1.5r)}}{r}$ can yield unstable results. Empirically, this provides a reason to favour $A_{indef}^{(c,r)}$ over other variants of the Nystr\"om method for symmetric indefinite matrices.

\begin{figure}[tbhp]
\hspace*{-0.6cm}
\subfloat[]{\label{fig:2a}\includegraphics[scale = 0.25]{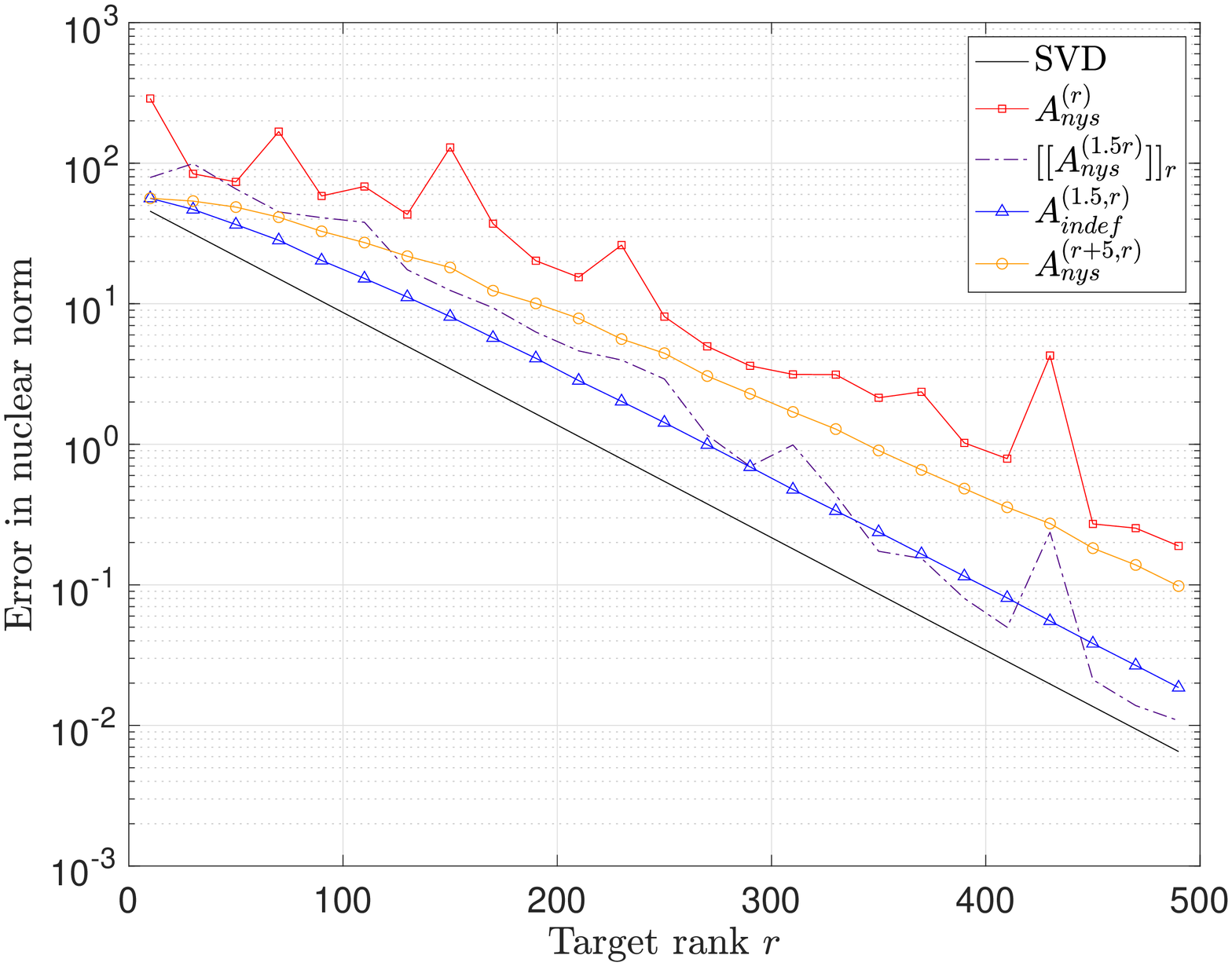}}
\hspace*{-0.5cm}
\subfloat[]{\label{fig:2b}\includegraphics[scale = 0.25]{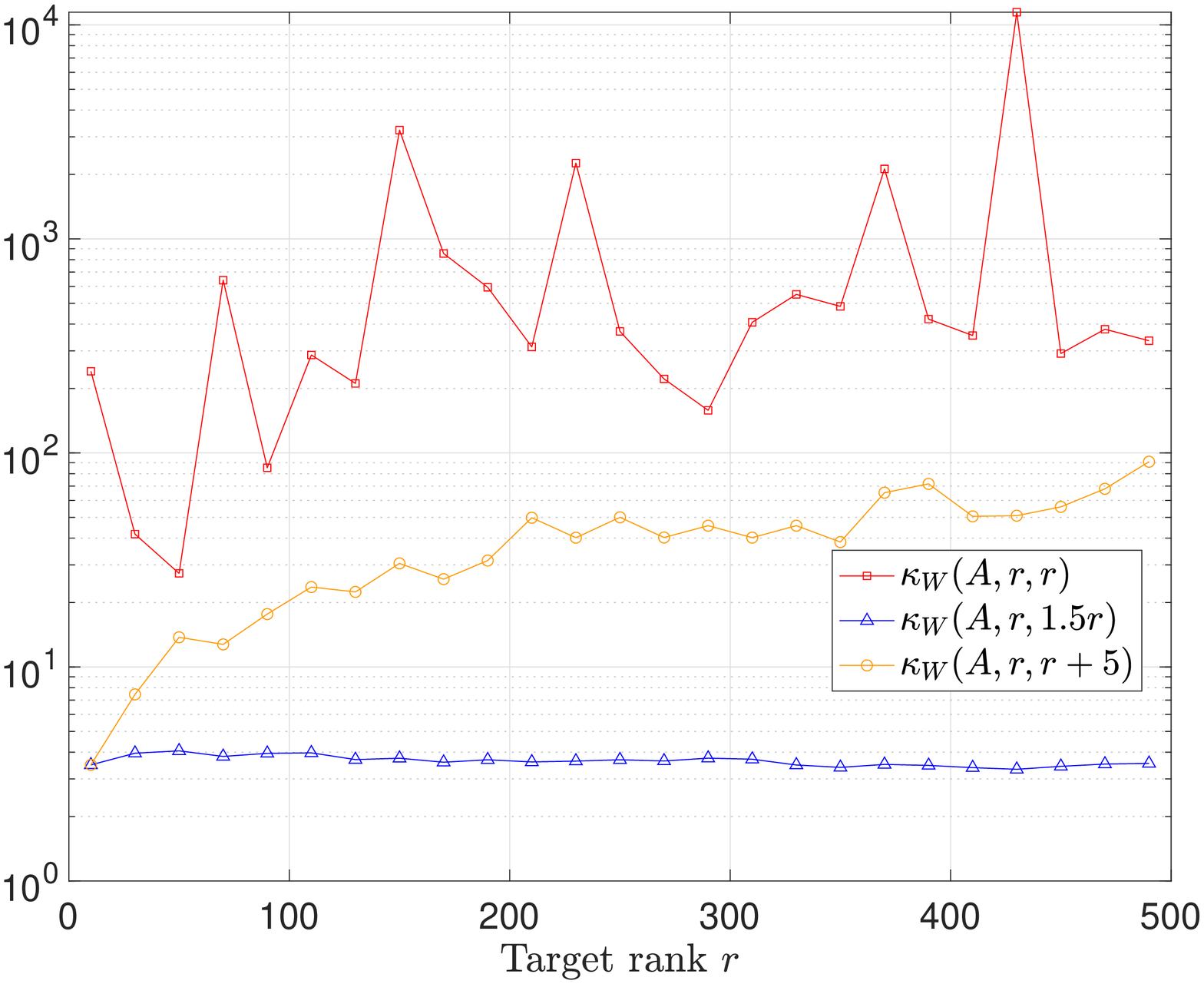}}
\centering 
\caption{Accuracy of the Nystr\"om approximations $A_{indef}^{(1.5,r)}$, $A_{nys}^{(r)}$, $A_{nys}^{(r+5,r)}$ and $\lowrank{A_{nys}^{(1.5r)}}{r}$ to a symmetric indefinite matrix $A\in \R^{1000\times 1000}$. Figure \ref{fig:2a} shows the Nystr\"om error in the nuclear norm and Figure \ref{fig:2b} shows the accuracy of the singular values of $W = X^T\!AX$ when compared with the singular values of $A$. We observe that the truncation in the core matrix $W$ can significantly increase the robustness and the accuracy of the Nystr\"om approximation.}
\label{Wcomp}
\end{figure}

\subsection{Random embeddings} \label{subsec:randemb}
A subspace embedding \cite{subem} is a linear map which preserves the $2$-norm of every vector in a given subspace, that is, $S\in \R^{s\times n}$ is a subspace embedding for the span of $A\in \R^{n\times n}$ with distortion $\epsilon \in (0,1)$ if
\begin{equation} \label{reldist}
    (1-\epsilon)\norm{Ax}_2 \leq \norm{SAx}_2 \leq (1+\epsilon)\norm{Ax}_2
\end{equation} for every $x\in \R^{n}$. A random embedding is a subspace embedding drawn at random that satisfy Equation \eqref{reldist} with high probability.

Random embeddings have more attractive properties than column sampling matrices \cite{nysprecond,martinsson_tropp_2020}, one of which is that the results obtained using random embeddings generally have smaller variance than the results obtained using column sampling. Below are few important examples of random embeddings.

\subsubsection{Gaussian matrices} \label{subsubsec:Gaussian}
A Gaussian embedding is a random matrix $G\in \R^{s\times n}$ with i.i.d. entries $G_{ij} \sim N(0,1/s)$. The scaling ensures that $\E[\norm{Gx}_2^2] = \norm{x}_2^2$ for every $x\in \R^n$. Gaussian embedding is the most widely used random embedding for theoretical analysis\footnote{Other random embeddings often lack strong theoretical guarantees, however they behave similarly to a Gaussian embedding in practice. For this reason, Gaussian theory is often used to provide a rule of thumb for the general behavior \cite{martinsson_tropp_2020}.} and often has optimal guarantees \cite{hmt,martinsson_tropp_2020}.  The cost of applying a Gaussian embedding to an $n\times n$ matrix is $O(n^2s)$. This becomes prohibitive for large $n$, so a more structured random embeddings are often used in practice.

\subsubsection{SRTTs}
A subsampled randomized trigonometric transform (SRTT) matrix is an $n\times s$ matrix with $n\geq s$ of the form \begin{equation}
    S = \sqrt{\frac{n}{s}}DFR^T
\end{equation} where $D\in \R^{n\times n}$ is a random diagonal matrix whose entries are independent and take $\pm 1$ with equal probability, $F\in \C^{n\times n}$ is a unitary trigonometric transform and $R\in \R^{s\times n}$ is a random restriction. In the complex case, $F$ is the unitary discrete Fourier transform (DFT) and in the real case, $F$ is commonly the discrete cosine transform (DCT). The sketch size needs to be $s = O(r\log r)$ for theoretical guarantees \cite{SRFTtropp}, but in practice $s = O(r)$ often suffices\footnote{For difficult examples, say a coherent example, the $\log r$ factor is necessary. (See Figure \ref{proposedalg})} \cite{hmt,martinsson_tropp_2020}. The cost of applying SRTT to an $n\times n$ matrix is $O(n^2\log r)$ \cite{coherence} using the subsampled FFT algorithm \cite{fastfft}.

\subsubsection{Sparse maps}
Sparse maps are sparse matrices with nonzero entries that are random signs \cite{clarksonwoodruff14,martinsson_tropp_2020,osnap,woodruff}. They are particularly useful for sparse data and they take the form
\begin{equation}
    S = \frac{1}{\sqrt{s}}[s_1,...,s_n] \in \R^{s\times n}
\end{equation} where the columns of $S$, the $s_i$'s are statistically independent and has exactly $\xi$ nonzero entries that take $\pm 1$ with equal probability, placed uniformly at random coordinates. We need the sketch size to be $s = O(r \log r)$ and the sparsity parameter to be $\xi = O(\log r)$ for theoretical guarantees \cite{cohen16}. In \cite{streamingtropp}, $\xi = \min\{s,8\}$ was recommended in practice. The cost of applying sparse maps to a matrix $A$ is $O(\xi \cdot nnz(A))$ where $nnz(A)$ is the number of nonzero entries of $A$ if sparse data structures and arithmetic are available. 

\subsection{Suggested algorithm}
For a general symmetric matrix $A\in \R^{n\times n}$ with the target rank $r$, we suggest
\begin{equation} \label{eq:propalg}
    A_{indef}^{(c,r)}  = AX \lowrank{X^T\!AX}{r}^\dagger (AX)^T = C\lowrank{W}{r}^\dagger C^T
\end{equation} where $X\in \R^{n\times s}$ is a random embedding with the sketch size $s= cr$ where $c>1$ is a modest constant. The algorithm is given in Algorithm \ref{alg:nys}. For the choice of random embeddings, if $A$ is sparse then we suggest sparse maps with sparsity $\xi = \min\{cr,8\}$ and when $A$ is dense we suggest SRTT matrices. The recommended sketch size is $s = 1.5r$ for efficiency, but if one wants a better approximation quality guarantee then the sketch size can be increased to, for example, $s = 2r$ or $s = 4r$. Note that the truncation is performed irrespectively of the singular values of $W$ (unlike previous studies, e.g. \cite{indefnys}); our analysis in Section~\ref{sec:thm} suggests that it is important that the number of singular values to be truncated $(s-r)=(c-1)r$ is proportional to $r$.

\begin{algorithm}[ht]\footnotesize
  \caption{Judiciously truncated Nystr\"om approximation for indefinite matrices}
  \label{alg:nys}
  \begin{algorithmic}[1]
\Require{Symmetric matrix $A\in \R^{n\times n}$, target rank $r<n$, sketch size $r<s<n$ (rec. $s = 1.5r$)}
\Ensure{$C\in \R^{n\times s}$ and $W_r^\dagger \in \R^{s\times s}$ with $\rank(W) \leq r$ as in \eqref{eq:propalg}}
\vspace{0.5pc}
\State Draw a random embedding $X\in \R^{n\times s}$ \Comment{Sparsity $\xi = \min\{s,8\}$ for sparse maps}
\State $C \gets AX$
\State $W \gets X^TC$
\State $[V,\Lambda] = \mathrm{eig}(W)$, eigendecomposition of $W$
\State $W_r^\dagger = V(:,1:r)\Lambda(1:r,1:r)^{\dagger}V(:,1:r)^T$, pseudoinverse of the best rank-$r$ approximation of $W$
\State Output $C\in \R^{n\times s}$ and $W_r^\dagger \in \R^{s\times s}$
\end{algorithmic}
\end{algorithm}

\paragraph{Complexity}
When a sparse map is used, the cost of Algorithm \ref{alg:nys} is $O(\xi\cdot nnz(A)+r^3)$ which consists of $O(\xi\cdot nnz(A))$ flops for forming the sketch and $O(r^3)$ flops for the eigendecomposition. With an SRTT sketch, the total cost is $O(n^2\log r +r^3)$, where $O(n^2\log r)$ is needed for forming the sketch and $O(r^3)$ for computing the eigendecomposition.\footnote{Since we are using random embeddings for robustness, Algorithm \ref{alg:nys} is strictly more expensive than classical Nystr\"om methods (column subsampling) if the columns can be sampled quickly.}

\paragraph{Eigendecomposition of $A_{indef}^{(c,r)}$} Algorithm \ref{alg:nys} as presented does not output the eigendecomposition of $A_{indef}^{(c,r)}$. To do this, we require an extra $O(nr^2 + r^3)$ flops. We need $O(nr^2)$ flops to compute the thin QR decomposition of $C = QR$, $O(r^3)$ flops to form and compute the eigendecomposition of $R\lowrank{W}{r}^\dagger R^T = U\Sigma U^T$ and $O(nr^2)$ flops to form $U_1 = QU$ giving us the eigendecomposition, $A_{indef}^{(c,r)} = U_1 \Sigma U_1^T$.

In Figure \ref{proposedalg}, we illustrate Algorithm \ref{alg:nys} for the SRFT sketch and the sparse map. The experiment was conducted with synthetic $2000\times 2000$ symmetric indefinite matrices. The top two plots have eigenvalues that decay geometrically from $1$ to $10^{-12}$ each assigned a random sign with equal probability and the eigenvectors are in a $2\times 2$ block diagonal form, $\mathrm{diag}(I_{200},U)$ where $I_{200}$ is the $200\times 200$ identity matrix and $U\in \R^{1800\times 1800}$ is a Haar distributed orthogonal matrix. This eigenvector matrix is a more coherent example than our previous examples and is known to be a difficult example for SRTT matrices \cite{coherence} (when the eigenvectors are Haar distributed, SRTT (or essentially any sketch) behaves the same as a Gaussian sketch, giving good results). The bottom two plots were generated using the same eigenvector matrix, but with eigenvalues equal to $\pm 1$ for the first $100$, $\pm 10^{-4}$ for the next $100$, $\pm 10^{-8}$ for the $100$ eigenvalues after that and $\pm 10^{-16}$ for the last $1700$ eigenvalues each assigned a random sign with equal probability. In the two left plots, we see that the SRFT sketch can fail if the sketch size is not large enough. This instability in the approximation can be fixed by enlarging the sketch size. We see that $s = r+5$ does not do well, but when $s = 4r$ the approximation becomes more accurate and robust. In the right plot, we see that the SRFT sketch with the sketch size $s = r\log r$, which comes with theoretical guarantees has excellent approximation quality. Finally, we see that the sparse map with sparsity $\xi = 8$ gives a robust approximation throughout, which can be improved by enlarging the sketch size.

\begin{figure}[htbp]
\hspace*{-1cm}
\includegraphics[scale = 0.38]{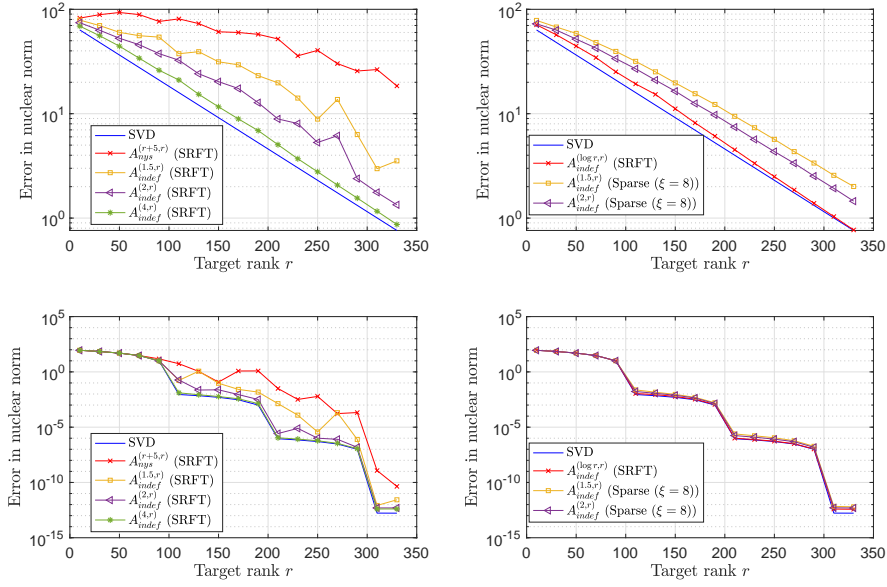}
\centering
\caption{Algorithm \ref{alg:nys}: A difficult (coherent) example for the SRFT sketch. The approximation can be unstable if the sketch size is too small for the SRFT sketch (left plots). This problem can be fixed by enlarging the sketch size. The right plots show that sparse maps have no issue with this example and the approximation is robust. 
}
\label{proposedalg}
\end{figure}


\section{Analysis}
\label{sec:thm}
For a general symmetric matrix $A\in \R^{n\times n}$, there are no known relative-error norm bounds for the Nystr\"om method. Here we show that for general symmetric matrices, the Nystr\"om method when used with a Gaussian sketch satisfies in expectation a relative-error nuclear norm bound under some orthogonal projection in the core matrix, when the singular values decay sufficiently fast. The analysis that follows establishes the accuracy not of Algorithm \ref{alg:nys}, but of a closely related variant of the Nystr\"om method.
The last paragraph of this section discusses this in more detail.

Let $A\in \R^{n\times n}$ be a symmetric matrix and let the eigendecomposition of $A$ be 
\begin{equation} \label{eigendecomp}
    A = V\Lambda V^T = [V_1,V_2,V_3] \begin{bmatrix}
    \Lambda_1 & 0 & 0 \\
    0 & \Lambda_2 & 0 \\
    0 & 0 & \Lambda_3
    \end{bmatrix} [V_1,V_2,V_3]^T
\end{equation} where $V\in \R^{n\times n}$ is the orthogonal eigenvector matrix of $A$ and $\Lambda \in \R^{n\times n}$ is a diagonal matrix containing the eigenvalues of $A$. The matrices with subscript $1$ have $r$ columns, those with subscript $2$ have $(c_1-1) r$ columns and subscript $3$ have $(n-c_1 r)$ columns where $r< c_1 r<n$ and $c_1>1$ is a constant such that $c_1r$ is a positive integer. The eigenvalues are ordered in non-increasing order with respect to their magnitude, so we have $\sigma_i(A) = |\lambda_i(A)|$ for all $i$.

Now we state our main theorem, and discuss the three key facts that will accompany our proof before getting to the proof immediately.

\begin{theorem} \label{mainthm}
Let $A \in \R^{n\times n}$ be a symmetric matrix as in \eqref{eigendecomp} and assume that $\lambda_r(A) \neq 0$. Let $c_1$ and $c_2$ be constants with $1<c_1 < c_2<\frac{n}{r}-1$ such that $c_1 r$ and $c_2 r$ are positive integers. Define $X_i:=V_i^TX$ for $i = 1,2,3$ where $X\in\mathbb{R}^{n\times c_2 r}$ is a Gaussian matrix, and set $B = X_3 Q_{\perp} (X_1 Q_{\perp})^\dagger~\in~\mathbb{R}^{(n-c_1 r)\times r}$ where $Q_{\perp}\in \mathbb{R}^{c_2 r \times (c_2-c_1+1)r}$ is an orthogonal complement of $X_2^T \in \mathbb{R}^{c_2 r \times (c_1-1)r}$. Let $(X_1 Q_{\perp})^\dagger = \hat{Q}\hat{R}$ be the thin QR decomposition of $(X_1 Q_{\perp})^\dagger$ and set $U := Q_{\perp}\hat{Q} \in \mathbb{R}^{c_2 r \times r}$. Then the orthogonal projector $P = UU^T \in \mathbb{R}^{c_2 r \times c_2 r}$ satisfies
\begin{equation} \label{thmbound}
    \mathbb{E}\left[\norm{E}_*|\Omega_F\right] \leq (1+\epsilon_{r,A}) \norm{A-\lowrank{A}{r}}_*
\end{equation} where 
\begin{equation}
E:= A-AX(PX^T\!AXP)^\dagger X^T\!A
\end{equation} is the associated Nystr\"om error, $\Omega_F$ is an event defined as
\begin{equation}
     \Omega_F:= \left\{\norm{|\Lambda_3|^{1/2}B}_F^2 \leq 0.5|\lambda_r(A)|  \right\}
\end{equation} where $|\Lambda_3|$ is defined element-wise and 
\begin{equation}
    \epsilon_{r,A} := 2b\sqrt{r} \left(1+\frac{|\lambda_{c_1 r+1}(A)|}{|\lambda_r(A)|} + \frac{2}{\sqrt{b}}\right)\frac{\norm{\Lambda_3}_*}{\norm{\Lambda_2}_*+\norm{\Lambda_3}_*}
\end{equation} where $b = \frac{r}{(c_2-c_1)r-1}$.
\end{theorem}

In the above theorem, $c_1$ and $c_2$ are oversampling factors which are of modest size, say $c_1 = 1.5$ and $c_2 = 2$. We need two factors because we need $X_1Q_{\perp}\in~\R^{r\times (c_2-c_1+1)r}$ and $X_3Q_{\perp}\in \R^{(n-c_1r)\times (c_2-c_1+1)r}$ to be \emph{rectangular} Gaussian matrices, which makes them well-conditioned with high probability \cite{ds01}. We can view $c_1$ as $c$ in Algorithm \ref{alg:nys} and $c_2$ to be the oversampling factor introduced to make the analysis possible. By making $c_1$, $c_2$ and $(c_2-c_1)$ larger, we can improve the bound in the above Theorem. The orthogonal projector $P = UU^T$ truncates the core matrix $W = X^T\!AX$ by removing the largest `unwanted' eigenvalues of $A$, i.e. the eigenvalues in $\Lambda_2$, using $X_\perp$ factor in $U$. This helps the core matrix to not be corrupted by the interaction between the target and the large `unwanted' singular values and singular vectors of $A$, which can happen when forming $X^T\!AX$. Lastly, the $\epsilon_{r,A}$ in the theorem plays a similar role to the distortion $\epsilon$ in Equation \eqref{relbound} and $\Omega_F$ is roughly the event that the eigenvalues of $A$ decay rapidly enough. If we assume that $A$ has a low-rank structure, for example, $|\lambda_r(A)|\gg |\lambda_{c_1r+1}(A)|$, then $\Omega_F$ would hold with high probability and $\epsilon_{r,A}$ would be a moderately-sized constant, which tells us that the relative-error nuclear norm bound in \eqref{thmbound} is good.

We now introduce three key facts that will be useful for our proof. The first fact follows closely the analysis in \cite{GN}. Let $\mathcal{P}:= \Lambda V^T X(PX^T\!AXP)^\dagger X^TV$ be an oblique projector. Then we can rewrite the associated Nystr\"om error as 
\begin{equation}
    E = V(I-\mathcal{P})\Lambda V^T.
\end{equation} 
As shown in \cite{GN}, it is straightforward to see that we can rewrite the associated Nystr\"om error as
\begin{equation}
    V^T E V = (I-\mathcal{P})\Lambda = (I-\mathcal{P})\Lambda (I-V^TXUM)
\end{equation} for any $M\in \R^{r\times n}$. Let $V_r = [I_r,0]^T \in \R^{n\times r}$ and set $M = (V_r^TV^TXU)^\dagger V_r^T$ then we get
\begin{equation} \label{idea1}
    V^TEV = (I-\mathcal{P})\Lambda (I-V_rV_r^T)(I-V^TXU(V_r^TV^TXU)^\dagger V_r^T).
\end{equation}
This modification of the associated Nystr\"om error will be important for our proof.

The second fact is the following. Let $f(x)$ be convex in the interval $[x_1,x_2]$ with $x_1<x_2$. Define $g(x)$ on $[x_1,x_2]$ to be the linear function joining the endpoints of $f$ on $[x_1,x_2]$, that is, $g(x) = \frac{f(x_2)-f(x_1)}{x_2-x_1}x+\frac{f(x_1)x_2-f(x_2)x_1}{x_2-x_1}$. Then $f(x)\leq g(x)$ on $[x_1,x_2]$. Let $Y$ be a random variable with $Y\in [x_1,x_2]$ almost surely. Then $f(Y) \leq g(Y)$ almost surely. Furthermore, if $Y\in [x_1,x_2]$ conditional on an event $\Omega$, then conditional on $\Omega$ we get
\begin{equation} \label{idea2}
    f(Y)\leq g(Y).
\end{equation}

The last fact is based on expected norm bounds for Gaussian matrices from \cite[App.~A]{hmt}. We can deduce the following lemma.
\begin{lemma} \label{idea3}
Let $B$ be the matrix as in Theorem \ref{mainthm} and let $S$ be a fixed real matrix such that $SB$ is defined. Then
\begin{equation}
    \E\norm{SB}_F^2 = b\norm{S}_F^2 
\end{equation} where $b=\frac{r}{(c_2-c_1)r-1}$ as in Theorem \ref{mainthm}.
\end{lemma}
\begin{proof}
Since conditional on $X_2$, $X_3 Q_{\perp}$ and $X_1 Q_{\perp}$ are two independent Gaussian matrices, we have
\begin{align*}
   \E_{X_1,Q_{\perp},X_3}\norm{SB}_F^2 &= \E_{X_1,Q_{\perp}} \left[\E_{X_3} \left[\norm{S X_3 Q_{\perp} (X_1 Q_{\perp})^\dagger}_F^2 \bigg|X_1,X_2 \right] \right]\\
   &= \norm{S}_F^2 E_{X_1,Q_{\perp}}\norm{(X_1 Q_{\perp})^\dagger}_F^2 \\ 
   &= \frac{r}{(c_2-c_1)r-1}\norm{S}_F^2
\end{align*}  using the tower property and the propositions in \cite[App.~A]{hmt}.
\end{proof}
Now using these three key facts we are ready to prove Theorem \ref{mainthm}.

\begin{proof}[Proof of Theorem \ref{mainthm}]
Since $U$ is an orthonormal matrix we have
\begin{align*}
         AX(PX^T\!AXP)^\dagger X^T\!A &= AX(UU^TX^T\!AXUU^T)^\dagger X^T\!A \\
         &= AXU(U^TX^T\!AXU)^\dagger U^TX^T\!A.
\end{align*}
Now since $X_1Q_{\perp}\in \mathbb{R}^{r \times (c_2-c_1+1) r}$ is a fat rectangular Gaussian matrix, hence full rank with probability $1$, we have $X_1Q_{\perp} (X_1Q_{\perp})^\dagger = I_{r}$. Therefore
\begin{equation}
    X_1Q_{\perp}\hat{Q} = \hat{R}^{-1}
\end{equation}
and we get
\begin{equation}
    V^TXU = \begin{bmatrix}
    X_1 \\ X_2 \\ X_3
    \end{bmatrix} Q_{\perp} \hat{Q} = 
    \begin{bmatrix}
     X_1 Q_{\perp}\hat{Q} \\ 0 \\ X_3 Q_{\perp}\hat{Q}
    \end{bmatrix} =  \begin{bmatrix}
    \hat{R}^{-1} \\ 0 \\ B\hat{R}^{-1}
    \end{bmatrix}
\end{equation} where $B = X_3 Q_{\perp} (X_1 Q_{\perp})^\dagger~\in~\mathbb{R}^{(n-c_1 r)\times r}$.

Now we use the first key fact (Equation \eqref{idea1}) and get
\begin{equation}
    V^TEV = (I-\mathcal{P})\Lambda (I-V_{r}V_{r}^T)(I-V^TXU(V_{r}^TV^TXU)^\dagger V_{r}^T)
\end{equation} where $\mathcal{P} = \Lambda V^T X(PX^T\!AXP)^\dagger X^TV$ and $V_r$ is as below.
Using
\begin{equation}
    V^TXU = \begin{bmatrix}
        \hat{R}^{-1} \\ 0 \\ B \hat{R}^{-1} 
    \end{bmatrix}, \Lambda = \begin{bmatrix}
        \Lambda_1 & 0 & 0 \\ 0 & \Lambda_2 & 0 \\ 0 & 0 & \Lambda_3 
    \end{bmatrix}, V_{r} = \begin{bmatrix}
        I_{r} \\ 0
    \end{bmatrix}
\end{equation} we get
\begin{align*}
    V^TE V &= (I-\mathcal{P})\Lambda \begin{bmatrix}
     0 \\ I_{n-r}
    \end{bmatrix}[0, I_{n-r}]\left(I- \begin{bmatrix}
        \hat{R}^{-1} \\ 0 \\ B\hat{R}^{-1}
    \end{bmatrix}
        \left(\hat{R}^{-1}\right)^\dagger [I_{r},0] \right) \\
    &= (I-\mathcal{P})\begin{bmatrix}
        0 & 0 & 0 \\ 0 & \Lambda_2 & 0 \\ -\Lambda_3 B & 0 & \Lambda_3
    \end{bmatrix}.
\end{align*}
We also get
\begin{align*}
    \mathcal{P} &= \begin{bmatrix}
        \Lambda_1 \hat{R}^{-1} \\ 0 \\ B\hat{R}^{-1}
    \end{bmatrix}\left( \begin{bmatrix}
        \hat{R}^{-1} \\ 0 \\ B\hat{R}^{-1}
    \end{bmatrix}^T \begin{bmatrix}
        \Lambda_1 \hat{R}^{-1} \\ 0 \\ \Lambda_3 B\hat{R}^{-1}
    \end{bmatrix} \right)^\dagger \begin{bmatrix}
        \hat{R}^{-1} \\ 0 \\ B\hat{R}^{-1}
    \end{bmatrix}^T \\
    &= \begin{bmatrix}
        \Lambda_1 \hat{R}^{-1} \\ 0 \\ \Lambda_3 B \hat{R}^{-1}
    \end{bmatrix}\left( \hat{R}^{-T}\left(\Lambda_1 +B^T\Lambda_3 B\right) \hat{R}^{-1} \right)^\dagger \begin{bmatrix}
        \hat{R}^{-1} \\ 0 \\ B \hat{R}^{-1}
    \end{bmatrix}^T \\
    &= \begin{bmatrix}
        \Lambda_1 \\ 0 \\ \Lambda_3 B
    \end{bmatrix} \left( \Lambda_1 + B^T\Lambda_3 B \right)^\dagger [I_r,0,B^T]
\end{align*} by taking out a factor of $\hat{R}^{-1}$ and $\hat{R}^{-T}$ from the pseudo-inverse. This is possible because if we condition on $\Omega_F$ then $(\Lambda_1 + B^T \Lambda_3 B)$ is a non-singular $r\times r$ matrix. Now for shorthand let $S := \Lambda_1 + B^T\Lambda_3 B $. Then
\begin{align*}
    I-\mathcal{P} = \begin{bmatrix}
        I_r - \Lambda_1 S^\dagger & 0 & -\Lambda_1 S^\dagger B^T \\
        0 & I_{(c_1-1)r} & 0 \\ -\Lambda_3B S^\dagger & 0 & I_{n-c_1 r}-\Lambda_3B S^\dagger B^T 
    \end{bmatrix}.
\end{align*}

Therefore
\begin{align*}
    V^TE V &= (I-\mathcal{P})\begin{bmatrix}
        0 & 0 & 0 \\ 0 & \Lambda_2 & 0 \\ -\Lambda_3 B & 0 & \Lambda_3
    \end{bmatrix} \\
    &= \begin{bmatrix}
    \Lambda_1 S^\dagger B^T\Lambda_3 B & 0 & -\Lambda_1 S^\dagger B^T\Lambda_3 \\ 0 & \Lambda_2 & 0 \\ -\Lambda_3 B+\Lambda_3B S^\dagger B^T \Lambda_3 B & 0 & \Lambda_3 - \Lambda_3B S^\dagger B^T \Lambda_3
    \end{bmatrix} \\
    &= \begin{bmatrix}
    \Lambda_1 S^\dagger B^T\Lambda_3 B & 0 & -\Lambda_1 S^\dagger B^T\Lambda_3 \\ 0 & \Lambda_2 & 0 \\ -\Lambda_3 B S^\dagger \Lambda_1  & 0 & \Lambda_3 - \Lambda_3B S^\dagger B^T \Lambda_3
    \end{bmatrix}.
\end{align*}
We now bound $E$ in the nuclear norm.
For shorthand, define the following
\begin{align*}
    a_1 &= \Lambda_1 S^\dagger B^T \Lambda_3 B \\
    a_2 &= \Lambda_1 S^\dagger B^T \Lambda_3 \\
    a_3 &= \Lambda_3 - \Lambda_3 B S^\dagger B^T\Lambda_3.
\end{align*}
We then have
\begin{align*}
    \norm{E}_* \leq \norm{\Lambda_2}_* + \norm{a_1}_*+2\norm{a_2}_*+\norm{a_3}_*.
\end{align*}
Let us note
\begin{equation}
   \Lambda_1 S^\dagger = \Lambda_1 \left(\Lambda_1 + B^T\Lambda_3 B\right)^\dagger = \left(I_r + B^T\Lambda_3 B \Lambda_1^{-1} \right)^\dagger
\end{equation}
conditional on $\Omega_F$ since $\lambda_r(A)\neq 0$ and
\begin{equation}
    \norm{B^T\Lambda_3 B}_F =  \norm{B^T|\Lambda_3|^{1/2} \sign{(\Lambda_3)} |\Lambda_3|^{1/2} B}_F \leq \norm{|\Lambda_3|^{1/2}B}_F^2
\end{equation} where $|\Lambda_3|$ and $\sign(\Lambda_3)$ are defined element-wise. 

We now bound $\E[\norm{a_1}_*|\Omega_F]$, $\E[\norm{a_2}_*|\Omega_F]$ and $\E[\norm{a_3}_*|\Omega_F]$ using the second (Equation \eqref{idea2}) and the third (Lemma \ref{idea3}) key fact. We start with $a_1$. Conditional on $\Omega_F$, we have
\begin{align*}
    \norm{a_1}_* &\leq \sqrt{r} \norm{\Lambda_1 S^\dagger B^T \Lambda_3 B}_F \\
    &\leq \sqrt{r} \norm{\left( I_r + B^T\Lambda_3 B\Lambda_1^{-1} \right)^\dagger}_2 \norm{B^T \Lambda_3 B}_F \\
    &\leq \sqrt{r} \frac{\norm{B^T \Lambda_3 B}_F}{1- \norm{B^T \Lambda_3 B}_F \norm{\Lambda_1^{-1}}_2} \\
    &\leq \frac{\sqrt{r}}{\norm{\Lambda_1^{-1}}_2} \frac{\norm{|\Lambda_3|^{1/2}B}_F^2 \norm{\Lambda_1^{-1}}_2}{1- \norm{|\Lambda_3|^{1/2}B}_F^2 \norm{\Lambda_1^{-1}}_2} \\
    &\leq \frac{\sqrt{r}}{\norm{\Lambda_1^{-1}}_2} \left(2 \norm{|\Lambda_3|^{1/2}B}_F^2\norm{\Lambda_1^{-1}}_2\right)
\end{align*}
\sloppy where the last inequality was obtained using the second fact with $Y = \norm{|\Lambda_3|^{1/2}B}_F^2\norm{\Lambda_1^{-1}}_2$, the event $\Omega_F$, the interval $[0,0.5]$, $f(x) = \frac{x}{1-x}$ which is convex on $[0,0.5]$ and $g(x) = 2x$. Now taking conditional expectation and using the third fact (Lemma \ref{idea3}) we get
\begin{align*}
    \mathbb{E}\left[\norm{a_1}_* |\Omega_F\right] &\leq 2\sqrt{r} \mathbb{E}\left[\norm{|\Lambda_3|^{1/2}B}_F^2 \middle\vert  \Omega_F  \right] \\
    &= 2\sqrt{r}b \norm{|\Lambda_3|^{1/2}}_F^2 \\
    &=2\sqrt{r}b\norm{\Lambda_3}_*.
\end{align*}
For $a_2$, it is similar to $a_1$. Conditional on $\Omega_F$ we have
\begin{align*}
    \norm{a_2}_* &\leq \sqrt{r}\norm{\Lambda_1 (\Lambda_1+B^T\Lambda_3 B)^\dagger B^T\Lambda_3}_F \\
    &\leq \sqrt{r}\norm{(I_r +B^T\Lambda_3 B \Lambda_1^{-1})^\dagger}_2\norm{|\Lambda_3|^{1/2}B}_F\norm{|\Lambda_3|^{1/2}}_F \\
    &\leq \frac{\sqrt{r} \sqrt{\norm{\Lambda_3}_*}}{\sqrt{\norm{\Lambda_1^{-1}}_2}}\frac{\norm{|\Lambda_3|^{1/2}B}_F\sqrt{\norm{\Lambda_1^{-1}}_2}}{1- \norm{|\Lambda_3|^{1/2}B}_F^2\norm{\Lambda_1^{-1}}_2} \\
    &\leq \frac{\sqrt{r} \sqrt{\norm{\Lambda_3}_*}}{\sqrt{\norm{\Lambda_1^{-1}}_2}} 2 \norm{|\Lambda_3|^{1/2}B}_F\sqrt{\norm{\Lambda_1^{-1}}_2}
\end{align*} where we used the second fact for the last inequality with $Y = \norm{|\Lambda_3|^{1/2}B}_F\sqrt{\norm{\Lambda_1^{-1}}_2}$, the interval $[0,\sqrt{0.5}]$, $f(x) = \frac{x}{1-x^2}$ and $g(x) = 2x$.  Therefore we get
\begin{align*}
    \mathbb{E}[\norm{a_2}_*|\Omega_F] &\leq 2\sqrt{r}\sqrt{\norm{\Lambda_3}_*} \mathbb{E} \left[\norm{|\Lambda_3|^{1/2}B}_F \middle\vert \Omega_F \right] \\
    &\leq 2\sqrt{r}\sqrt{\norm{\Lambda_3}_*}\sqrt{\mathbb{E} \left[\norm{|\Lambda_3|^{1/2}B}_F^2 \middle\vert \Omega_F \right]} \\
    &\leq 2\sqrt{rb}\norm{\Lambda_3}_*
\end{align*} using Lemma \ref{idea3}.

Finally for $a_3$, we get
\begin{equation*}
    \norm{a_3}_* \leq \norm{\Lambda_3}_* + \sqrt{r}\norm{|\Lambda_3|^{1/2}}_2^2 \norm{|\Lambda_3|^{1/2}B}_F^2 \norm{(\Lambda_1+B^T\Lambda_3 B)^\dagger}_2
\end{equation*} in a similar manner, and conditional on $\Omega_F$ we have 
\begin{align*}
     \norm{|\Lambda_3|^{1/2}B}_F^2 \norm{(\Lambda_1+B^T\Lambda_3 B)^\dagger}_2 &\leq \norm{|\Lambda_3|^{1/2}B}_F^2 \norm{\Lambda_1^{-1}}_2\norm{\Lambda_1 (\Lambda_1+B^T\Lambda_3 B)^\dagger}_2\\
    &\leq \frac{\norm{|\Lambda_3|^{1/2}B}_F^2 \norm{\Lambda_1^{-1}}_2}{1-\norm{|\Lambda_3|^{1/2}B}_F^2 \norm{\Lambda_1^{-1}}_2} \\
    &\leq 2\norm{|\Lambda_3|^{1/2}B}_F^2 \norm{\Lambda_1^{-1}}_2
\end{align*} using the second fact with the same values as the $a_1$ case. Therefore
\begin{align*}
    \mathbb{E}[\norm{a_3}_*|\Omega_2] &\leq \norm{\Lambda_3}_* + 2\sqrt{r}\norm{\Lambda_3}_2 \norm{\Lambda_1^{-1}}_2 \mathbb{E}\left[\norm{|\Lambda_3|^{1/2}B}_F^2 \middle\vert \Omega_F \right] \\
    &\leq \norm{\Lambda_3}_* + 2\sqrt{r}b \norm{\Lambda_3}_2 \norm{\Lambda_1^{-1}}_2 \norm{\Lambda_3}_*.
\end{align*}

Finally, combining everything together we get
\begin{equation}
    \mathbb{E}\left[\norm{E}_*|\Omega_F\right] \leq \norm{\Lambda_2}_* +\norm{\Lambda_3}_* + 2b\sqrt{r} \left(1+\frac{|\lambda_{c_1 r+1}(A)|}{|\lambda_r(A)|}  + \frac{2}{\sqrt{b}}\right)\norm{\Lambda_3}_*.
\end{equation}
Therefore 
\begin{equation}
    \mathbb{E}\left[\norm{E}_*|\Omega_F\right] \leq (1+\epsilon_{r,A})\left(\norm{\Lambda_2}_*+\norm{\Lambda_3}_*\right) =(1+\epsilon_{r,A}) \norm{A-\lowrank{A}{r}}_*
\end{equation}
with 
\begin{equation}
    \epsilon_{r,A} = 2b\sqrt{r} \left(1+\frac{|\lambda_{c_1 r+1}(A)|}{|\lambda_r(A)|}  + \frac{2}{\sqrt{b}}\right)\frac{\norm{\Lambda_3}_*}{\norm{\Lambda_2}_*+\norm{\Lambda_3}_*}.
\end{equation}
\end{proof}
\begin{remark}\ \label{remark:thm}
\begin{enumerate}
    \item The relative-error nuclear norm bound is informative if $\epsilon_{r,A}$ is small. Now since $b \approx (c_2-c_1)^{-1} = O(1)$, we have
    \begin{equation} \label{epsbound}
        \epsilon_{r,A} = O\left(\frac{\sqrt{r}\norm{\Lambda_3}_*}{\norm{\Lambda_2}_*+\norm{\Lambda_3}_*}\right).
    \end{equation} Therefore the relative-error nuclear norm bound is good if 
    \begin{equation}
        \sqrt{r} \sum_{j = c_1r+1}^n |\lambda_j(A)| = \sqrt{r}\norm{\Lambda_3}_* \lesssim \norm{\Lambda_2}_* =  \sum_{j = r+1}^{c_1r} |\lambda_j(A)|.
    \end{equation}
    \item Using a similar proof technique we can obtain mixed norm bounds. The 2-norm version of Theorem \ref{mainthm} would give 
    \begin{equation}
    \mathbb{E}\left[\norm{E}_2|\Omega_F\right] \leq \norm{A-\lowrank{A}{r}}_2+ \frac{\epsilon_{r,A}}{\sqrt{r}}\norm{A-\lowrank{A}{r}}_*
    \end{equation}
    and the Frobenius norm version would give 
    \begin{equation}
    \mathbb{E}\left[\norm{E}_F|\Omega_F\right] \leq \norm{A-\lowrank{A}{r}}_F+ \frac{\epsilon_{r,A}}{\sqrt{r}}\norm{A-\lowrank{A}{r}}_*
    \end{equation}
    where $\epsilon_{r,A}$ is as in Theorem \ref{mainthm}. Therefore the constant in front of the best rank-$r$ nuclear norm error improves to $\epsilon_{r,A}/\sqrt{r} = O(1)$ using the second remark \eqref{epsbound}.
This type of mixed norm bounds along with the relative-error nuclear norm bound in Theorem \ref{mainthm} are fairly consistent with the SPSD versions in Table $1$ of \cite{GM}.
    \item We can relax the condition $\Omega_F$ to $\Omega_2:=\left\{\norm{|\Lambda_3|^{1/2}B}_2^2 \leq 0.5 |\lambda_r(A)|\right\}$ at the cost of a slightly worse bound in Equation \eqref{thmbound}. It is easy to show that the bound in Equation \eqref{thmbound} then changes to
    \begin{equation}
    \mathbb{E}\left[\norm{E}_*|\Omega_2\right] \leq (1+\sqrt{r}\epsilon_{r,A}) \norm{A-\lowrank{A}{r}}_*.
\end{equation}   
\end{enumerate}
\end{remark}

\paragraph{Probability of $\Omega_F$} The probability of the event $\Omega_F$ happening can be computed by following the proof of Theorem 10.8 in \cite{hmt} using $k = r$ and $p = (c_2-c_1)r$ and Lemma \ref{idea3}. We get
\begin{align*}
    \Prob\left(\norm{|\Lambda_3|^{1/2}B}_F \leq \sqrt{\norm{\Lambda_3}_*}\sqrt{\frac{3r}{(c_2-c_1)r+1}}t + \sqrt{\norm{\Lambda_3}_2} \frac{e\sqrt{(c_2-c_1+1)r}}{(c_2-c_1)r+1} tu\right) \\ \geq 1-2t^{-(c_2-c_1)r}-e^{-u^2/2}
\end{align*} for $u,t >0$. Now using $(x+y)^2 \leq 2(x^2+y^2)$, we get
\begin{align*}
    \left(\sqrt{\norm{\Lambda_3}_*}\sqrt{\frac{3r}{(c_2-c_1)r+1}}t + \sqrt{\norm{\Lambda_3}_2} \frac{e\sqrt{(c_2-c_1+1)r}}{(c_2-c_1)r+1} tu \right)^2 \\ \leq 2t^2\left(\norm{\Lambda_3}_* \frac{3r}{(c_2-c_1)r+1} + \norm{\Lambda_3}_2 \frac{e^2(c_2-c_1+1)r}{\left((c_2-c_1)r+1\right)^2} u^2\right).
\end{align*} Therefore
\begin{equation}
    \Prob\left(\Omega_F\right) = \Prob\left(\norm{|\Lambda_3|^{1/2}B}_F^2 \leq 0.5|\lambda_r(A)| \right) \geq 1-2t^{-(c_2-c_1)r}-e^{-u^2/2}
\end{equation} if
\begin{equation}
    0.5|\lambda_r(A)| \geq 2t^2\left(\norm{\Lambda_3}_* \frac{3r}{(c_2-c_1)r+1} + \norm{\Lambda_3}_2 \frac{e^2(c_2-c_1+1)r}{\left((c_2-c_1)r+1\right)^2} u^2\right),
\end{equation}i.e., $\Omega_F$ holds with high probability when the tail singular values of $A$ decay rapidly. A similar result can also be derived for $\Omega_2$ by following the same results in \cite{hmt}.

\paragraph{Mixed norm bounds} We can obtain mixed norm bounds for Theorem \ref{mainthm}. The 2-norm version of Theorem \ref{mainthm} would give 
    \begin{equation} \label{2normmix}
    \mathbb{E}\left[\norm{E}_2|\Omega_F\right] \leq \norm{A-\lowrank{A}{r}}_2+ \frac{\epsilon_{r,A}}{\sqrt{r}}\norm{A-\lowrank{A}{r}}_*
    \end{equation}
    and the Frobenius norm version would give 
    \begin{equation} \label{frobnormmix}
    \mathbb{E}\left[\norm{E}_F|\Omega_F\right] \leq \norm{A-\lowrank{A}{r}}_F+ \frac{\epsilon_{r,A}}{\sqrt{r}}\norm{A-\lowrank{A}{r}}_*
    \end{equation}
    where $\epsilon_{r,A}$ is as in Theorem \ref{mainthm}. This improves the constant in front of the best rank-$r$ nuclear norm error to $\epsilon_{r,A}/\sqrt{r} = O(1)$ using the first remark \eqref{epsbound} in Remark \ref{remark:thm}. The proof for the two mixed norm bounds above can be obtained by following the proof of Theorem \ref{mainthm}. More specifically, the proof for the mixed norm bounds stay the same until we bound $a_1$, $a_2$ and $a_3$. To get the mixed norm bound, we use the appropriate norms to bound $a_1$, $a_2$ and $a_3$. For example, to bound $\norm{a_1}_F$, we start similarly as in the nuclear norm case by conditioning on $\Omega_F$ to obtain
    \begin{align*}
        \norm{a_1}_F
    &\leq \norm{\left( I_r + B^T\Lambda_3 B\Lambda_1^{-1} \right)^\dagger}_2 \norm{B^T \Lambda_3 B}_F \\
    &\leq \frac{\norm{B^T \Lambda_3 B}_F}{1- \norm{B^T \Lambda_3 B}_F \norm{\Lambda_1^{-1}}_2} \\
    &\leq \frac{1}{\norm{\Lambda_1^{-1}}_2} \left(2 \norm{|\Lambda_3|^{1/2}B}_F^2\norm{\Lambda_1^{-1}}_2\right).
    \end{align*} We then get 
    \begin{align*}
    \mathbb{E}\left[\norm{a_1}_* |\Omega_F\right] &\leq 2 \mathbb{E}\left[\norm{|\Lambda_3|^{1/2}B}_F^2 \middle\vert  \Omega_F  \right] =2b\norm{\Lambda_3}_*.
    \end{align*}
    The bound for $\norm{a_2}_F,\norm{a_3}_F, \norm{a_1}_2,\norm{a_2}_2$ and $\norm{a_3}_2$ follows similarly. The mixed norm bounds \eqref{2normmix} and \eqref{frobnormmix} along with the relative-error nuclear norm bound in Theorem \ref{mainthm} are fairly consistent with the SPSD versions in Table $1$ of \cite{GM}.

Theorem \ref{mainthm} and its proof cannot simply be translated into an algorithm because the proof relies on the eigendecomposition of $A$, which is too expensive to compute. However, the proof naturally suggests Algorithm \ref{alg:nys}. From the proof of Theorem \ref{mainthm}, under the condition that the matrix has a low-rank structure discussed in this section, for example in the paragraph after the statement of Theorem \ref{mainthm} or in the remark above, we have that a projection is desired in the core matrix. This projection gets rid of the large `unwanted' eigenvalues of $A$, i.e. the eigenvalues in $\Lambda_2$. In the Nystr\"om method, a natural analogue is to truncate the smallest few singular values in the core matrix $W = X^T\!AX$ to achieve the target rank $r$, which is what has been done in Algorithm \ref{alg:nys}. The theorem also suggests that the sketch size should be proportional to the target rank $r$, which is what we suggest in Algorithm \ref{alg:nys}. Despite Algorithm \ref{alg:nys} lacking complete theory (even for the SPSD case), we suggest it because the algorithm does seem to work well in practice as we illustrate below.

\section{Numerical illustration} \label{sec:numexp}
We first illustrate Theorem \ref{mainthm} and Algorithm \ref{alg:nys} through experiments. In Figure \ref{thmfig}, we show a priori and a posteriori error in Theorem \ref{mainthm}, and Algorithm \ref{alg:nys} using $1000\times 1000$ symmetric indefinite matrices. In the left plot, the matrix $A$ has eigenvalues that decay geometrically from $1$ to $10^{-12}$ each assigned a random sign with equal probability. In the right plot, $A$ has eigenvalues equal to $\pm 1$ for the first $100$ eigenvalues and $\pm 10^{-10}$ for the other $900$ eigenvalues each assigned a random sign with equal probability; this example illustrates the performance when there is a gap in the singular values. The eigenvectors for both plots are in a $2\times 2$ block diagonal form, $\mathrm{diag}(I_{100},U)$ where $I_{100}$ is the $100\times 100$ identity matrix and $U\in\R^{900\times 900}$ is a Haar distributed orthogonal matrix. Both the algorithm and the theorem were constructed using the Gaussian sketch with the sketch size $1.5r$ for the algorithm and $c_1r = 1.5r$ and $c_2r = 2r$ for the theorem. We see that $\Omega_F$ holds whenever there is a rapid decay of eigenvalues, i.e., when $|\lambda_r| \gg |\lambda_{c_1 r+1}|$. But more importantly, we see that the bound holds when the event $\Omega_F$ occurs (circles) and frequently holds even if the event $\Omega_F$ did not occur (crosses). The theorem does extremely well when $\Omega_F$ has occurred. We see that the algorithm gives a good robust approximation that is a modest factor worse than the best approximation given by the SVD. Although the theorem does better than the algorithm when $\Omega_F$ holds, the theorem can give unstable approximation when $\Omega_F$ does not hold. This illustrates that the algorithm, which arose from the theorem, works well in practice. 

\begin{figure}[htbp]
\hspace*{-1cm}
\includegraphics[scale = 0.51]{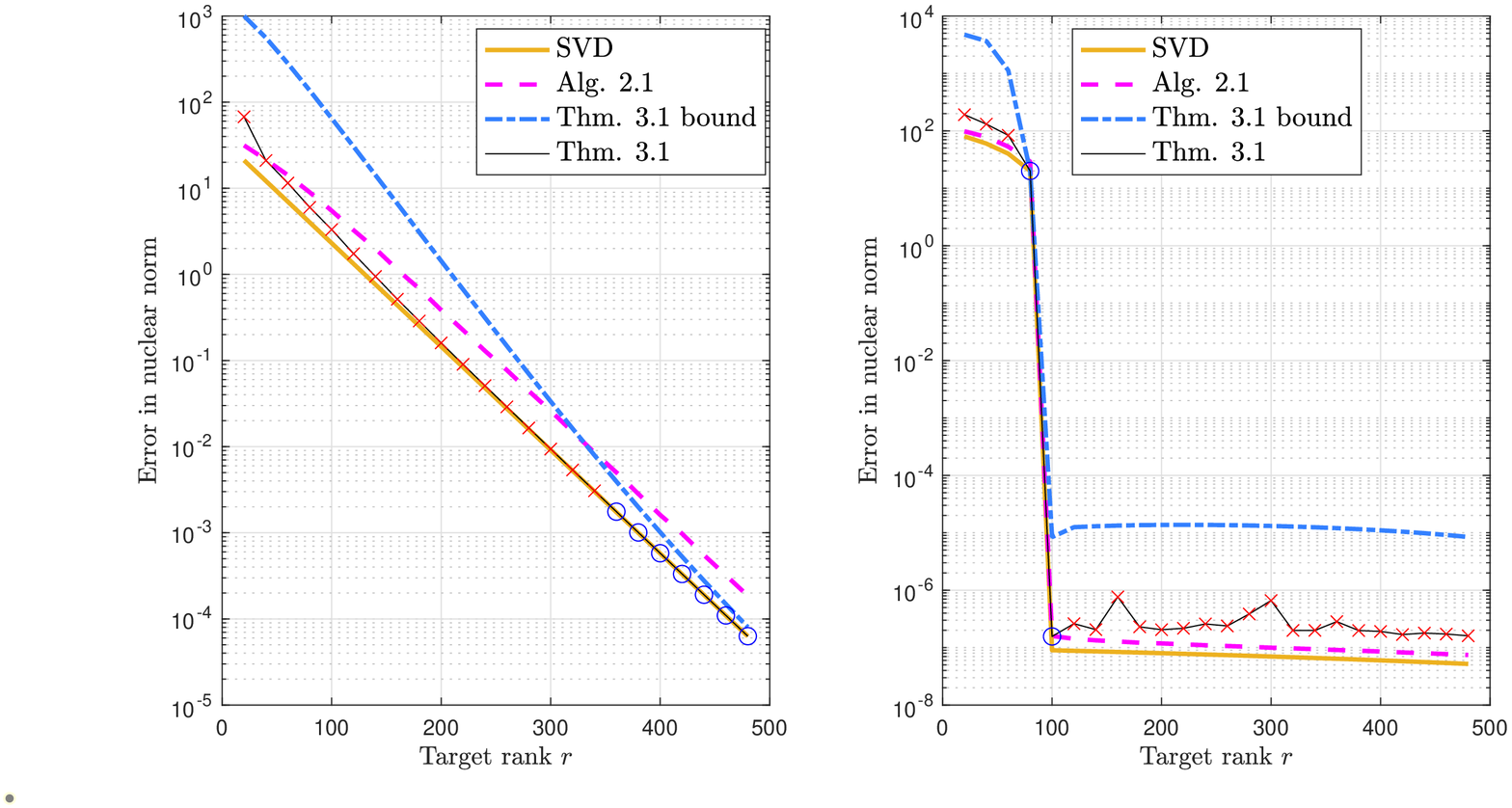}
\centering
\caption{Two plots showing the empirical results for Theorem \ref{mainthm} and Algorithm \ref{alg:nys}. Algorithm \ref{alg:nys} is robust with the approximation being a modest factor worse than the best approximation. Theorem \ref{mainthm} bound holds when $\Omega_F$ has occurred (circles on Theorem \ref{mainthm}) and also frequently holds even it $\Omega_F$ has not occurred (crosses on Theorem \ref{mainthm}). Theorem \ref{mainthm} does extremely well when $\Omega_F$ has occurred.}
\label{thmfig}
\end{figure}

In experiments not shown here, we 
compared Algorithm~\ref{alg:nys} with randomized SVD \cite{hmt} and the generalized Nystr\"om method \cite{clarksonwoodruff14,GN,TroppGN,fastfft}, which are applicable to nonsymmetric (and rectangular) matrices and do not preserve symmetry. 
We observe that Algorithm~\ref{alg:nys} tends to obtain a slightly better approximant for a fixed rank $r$.

\subsection{Synthetic examples}
We now compare some of the existing algorithms against Algorithm \ref{alg:nys} using different kernel functions and synthetic dataset. We illustrate the following algorithms
\begin{enumerate}
    \item Algorithm \ref{alg:nys} with the SRFT sketch and the sketch size $s = 2r$,
    \item Algorithm \ref{alg:nys} with uniform column sampling and the sketch size $s = 2r$,
    \item Algorithm \ref{alg:nys} with leverage score column sampling and the sketch size $s = 2r$,
    \item Submatrix-Shifted (SMS) Nystr\"om \cite{smsnystrom} with uniform column sampling and $s_1 = r$, $s_2 = 2r$ and $\alpha = 1.5$,
    \item Submatrix-Shifted (SMS) Nystr\"om \cite{smsnystrom} with the Gaussian sketch and $s_1 = r$, $s_2 = 2r$ and $\alpha = 1.5$,
    \item Stabilized Nystr\"om \cite{indefnys} with the SRFT sketch, $s=r$ and $\epsilon = 10^{-14}$
\end{enumerate} where $r$ is the target rank and the parameters for SMS Nystr\"om and Stabilized Nystr\"om are as recommended in their original papers.\footnote{For stabilized Nystr\"om method, $s=r$ was chosen to ensure that all approximations in the experiment have rank at most $r$ and $\epsilon = 10^{-14}$ as suggested in the original paper was chosen to try diminish the error that might come from taking the pseudo-inverse of the core matrix $W$.} For SMS Nystr\"om method, the Gaussian sketch was not used in the original paper \cite{smsnystrom}. We use the following kernel functions
\begin{enumerate}
    \item Epanechnikov kernel: $k_1(x,y) = \max\{1-\norm{x-y}^2,0\}$
    \item Multiquadric kernel: $k_2(x,y) = \sqrt{1+\norm{x-y}^2}$
    \item Thin plate spline: $k_3(x,y) = \norm{x-y}^2 \ln\left(\norm{x-y}^2\right)$
\end{enumerate}
to generate the kernel matrices. The kernel matrices $K^{(1)},K^{(2)}$ and $K^{(3)}$ corresponding to the kernel functions $k_1,k_2$ and $k_3$ were generated by sampling $1000$ random numbers  $\{x_i\}_{i=1}^{1000}$ from the standard normal distribution, i.e., $K^{(\ell)}_{ij} = k_\ell(x_i,x_j)$. All the kernel matrices are symmetric indefinite.

In Figure \ref{indefkerfig}, we illustrate the results. The eigenvalue histogram is shown in the left plots. The right plots show the approximation. We see that SMS Nystr\"om performs poorly in all $3$ examples except the Gaussian case for the multiquadric kernel. This is possibly because the extreme eigenvalues are large in magnitude so the large shift is ruining the approximation quality. The stabilized Nystr\"om works well for the multiquadric kernel and the thin plate spline, but the approximation is very unstable for the Epanechnikov kernel. This is possibly because the number of positive and the negative eigenvalues are about the same with similar magnitudes for the Epanechnikov kernel, which can increase the chance of instability in the core matrix.\footnote{To our knowledge, the numerical behavior of 
stabilized Nystr\"om method is an open problem; the stability analysis in \cite{GN} applies only to an algorithm where $A$ is sketched from both sides using independent sketches of different dimensions.} This also tells us that the truncation in the core matrix should not depend on the magnitude of the singular values of $W$, but the truncation should always happen proportional to the target rank. Algorithm \ref{alg:nys} using uniform column sampling and leverage score column sampling are both unstable for all $3$ examples, which shows the unreliability of using column sampling matrices. On the other hand, Algorithm \ref{alg:nys} using the SRFT sketch works well in all cases.

\begin{figure}[!ht]
\hspace*{-0.5cm}
\subfloat[]{\label{a}\includegraphics[scale = 0.25]{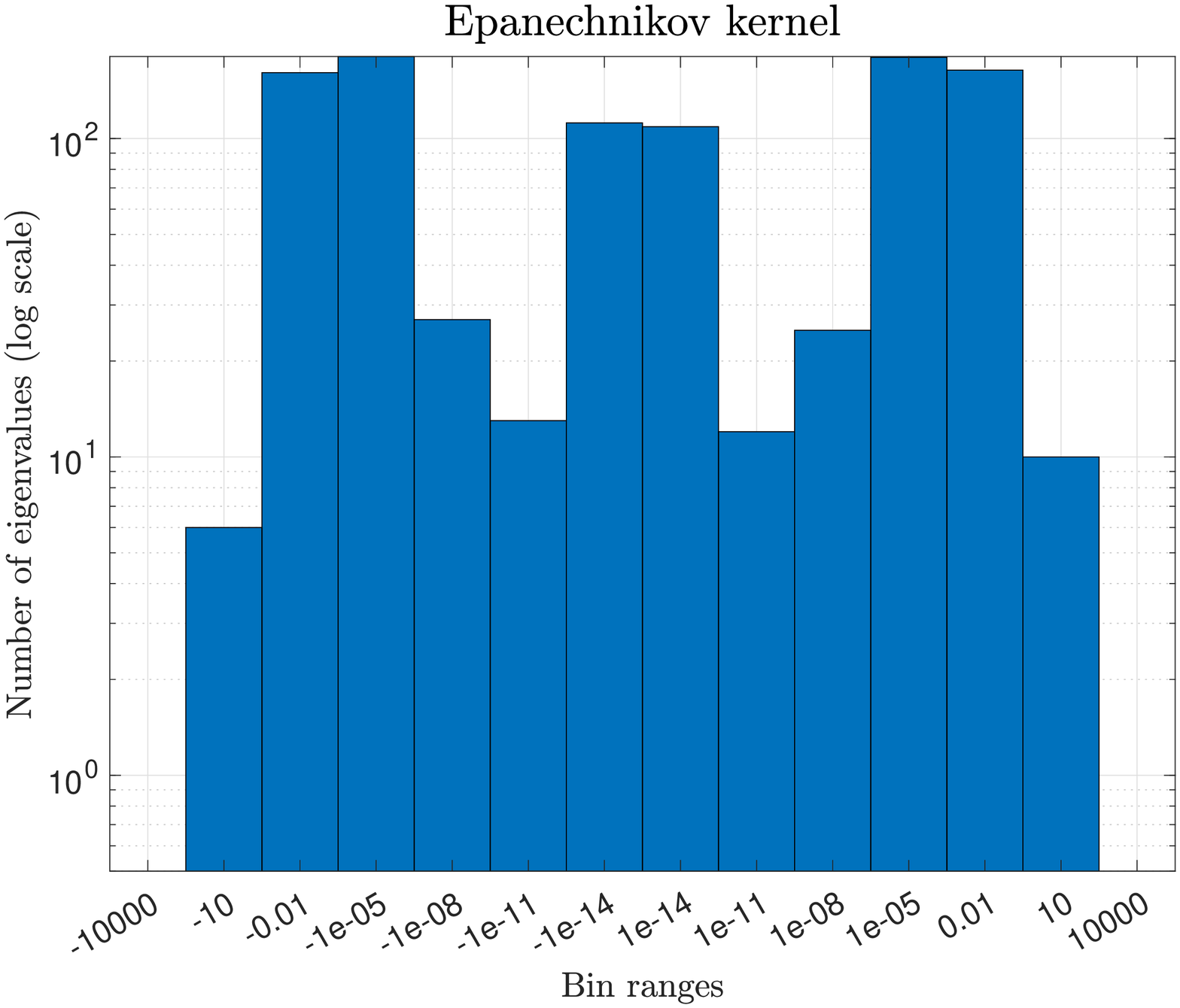}}
\hspace*{-0.7cm}
\subfloat[]{\label{b}\includegraphics[scale = 0.25]{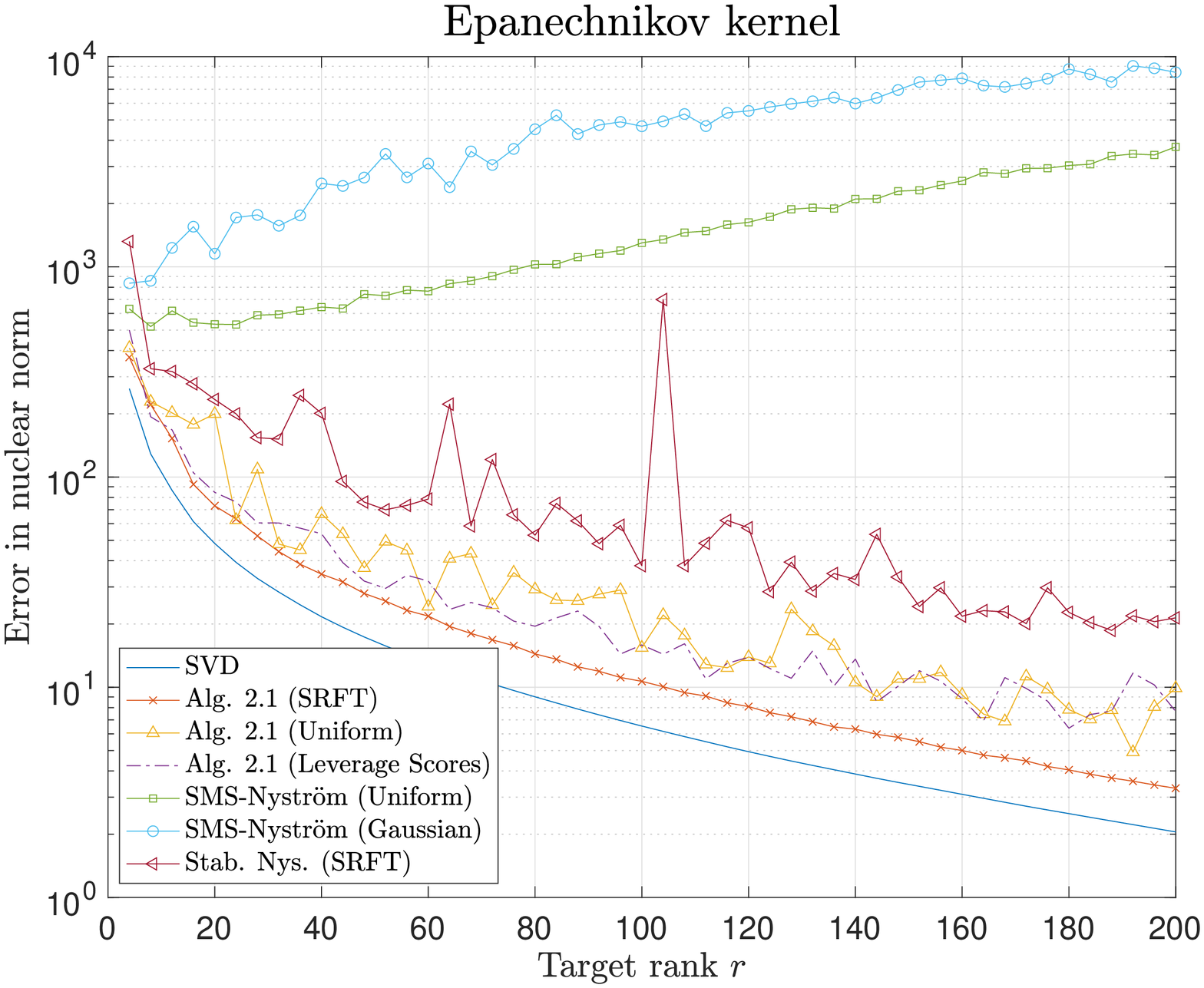}} \\ \vspace{0.5cm}
\hspace*{-0.5cm}
\subfloat[]{\label{c}\includegraphics[scale = 0.25]{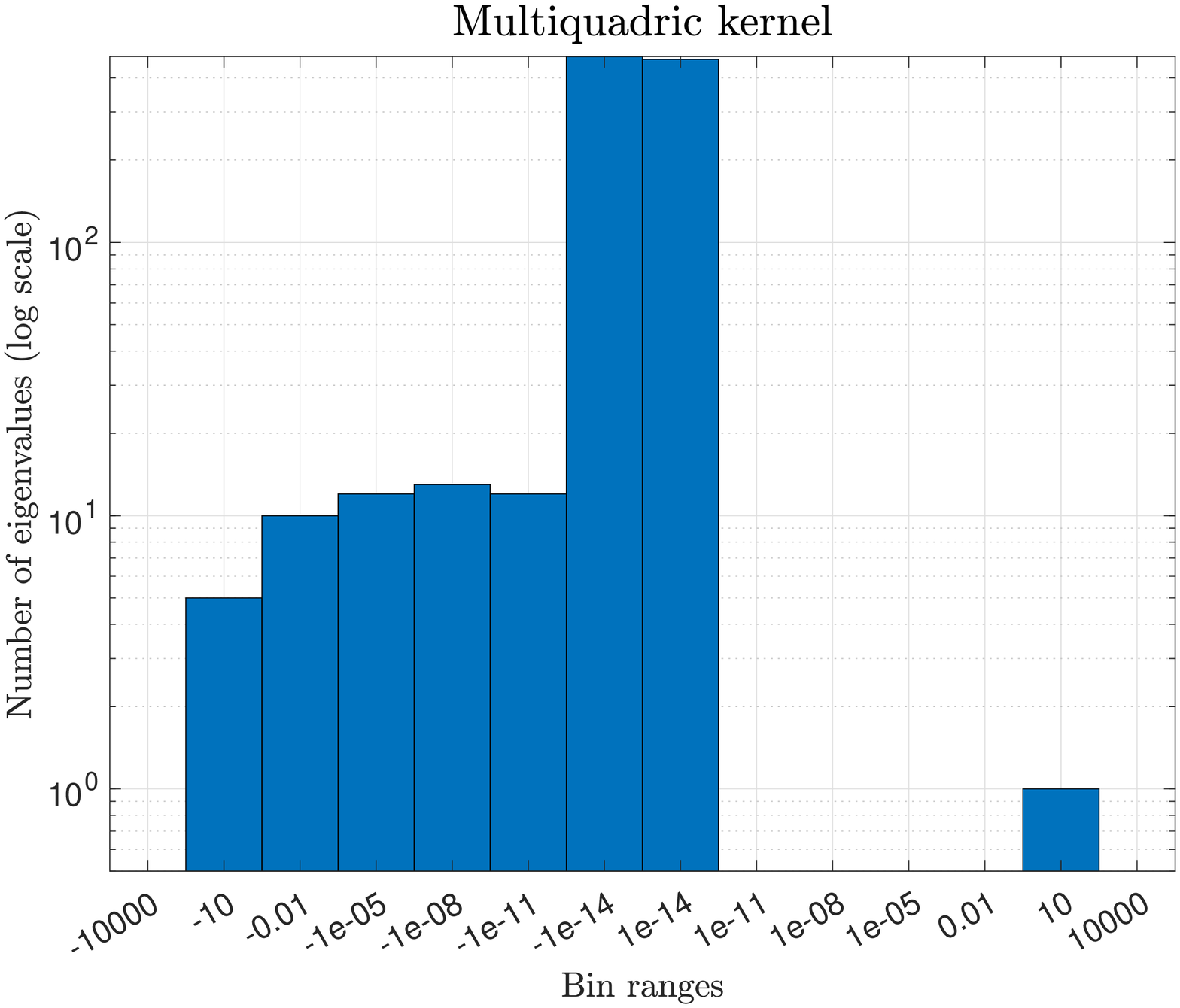}}
\hspace*{-0.7cm}
\subfloat[]{\label{d}\includegraphics[scale = 0.25]{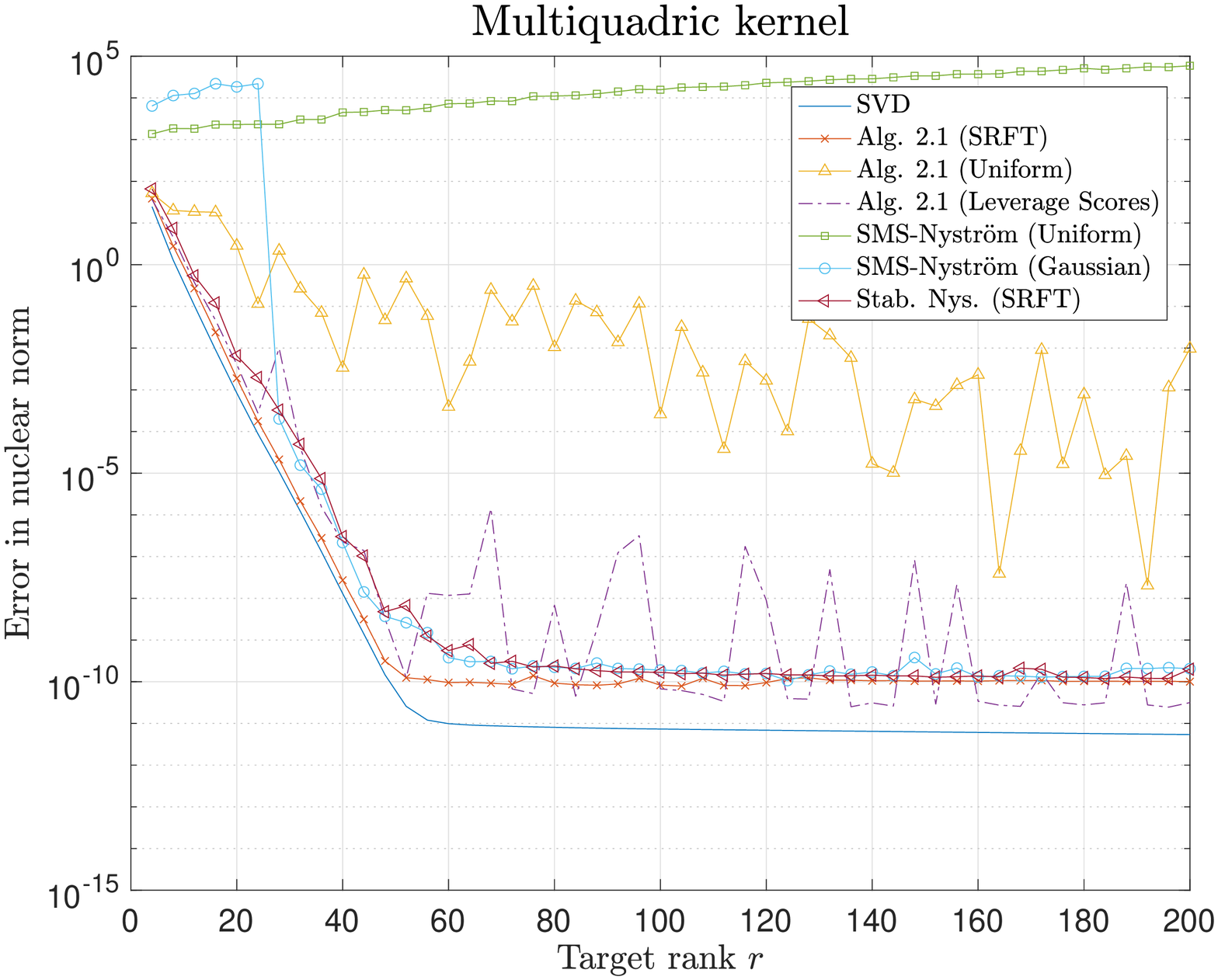}} \\ 
\hspace*{-0.5cm}
\subfloat[]{\label{e}\includegraphics[scale = 0.25]{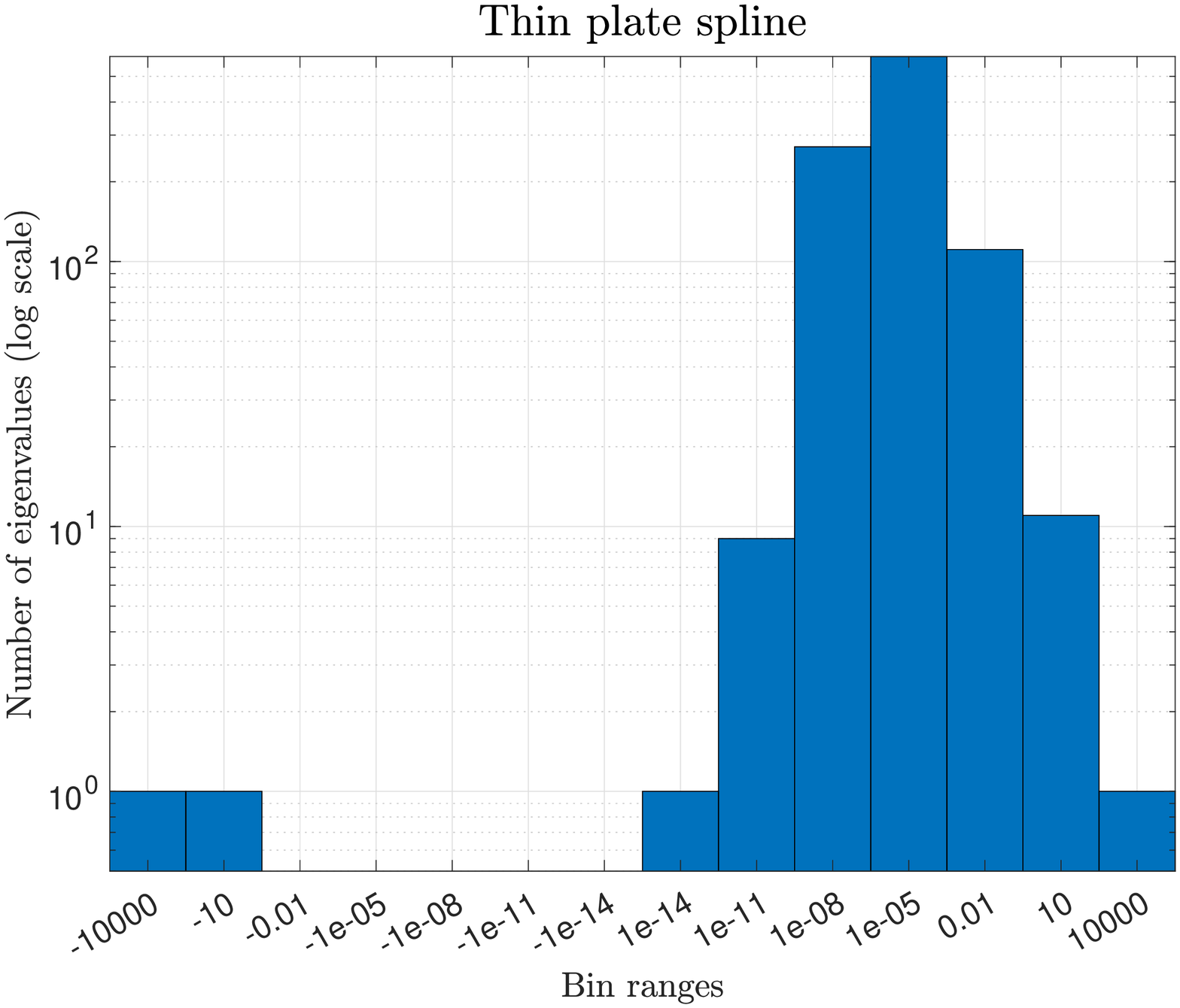}}
\hspace*{-0.7cm}
\subfloat[]{\label{f}\includegraphics[scale = 0.25]{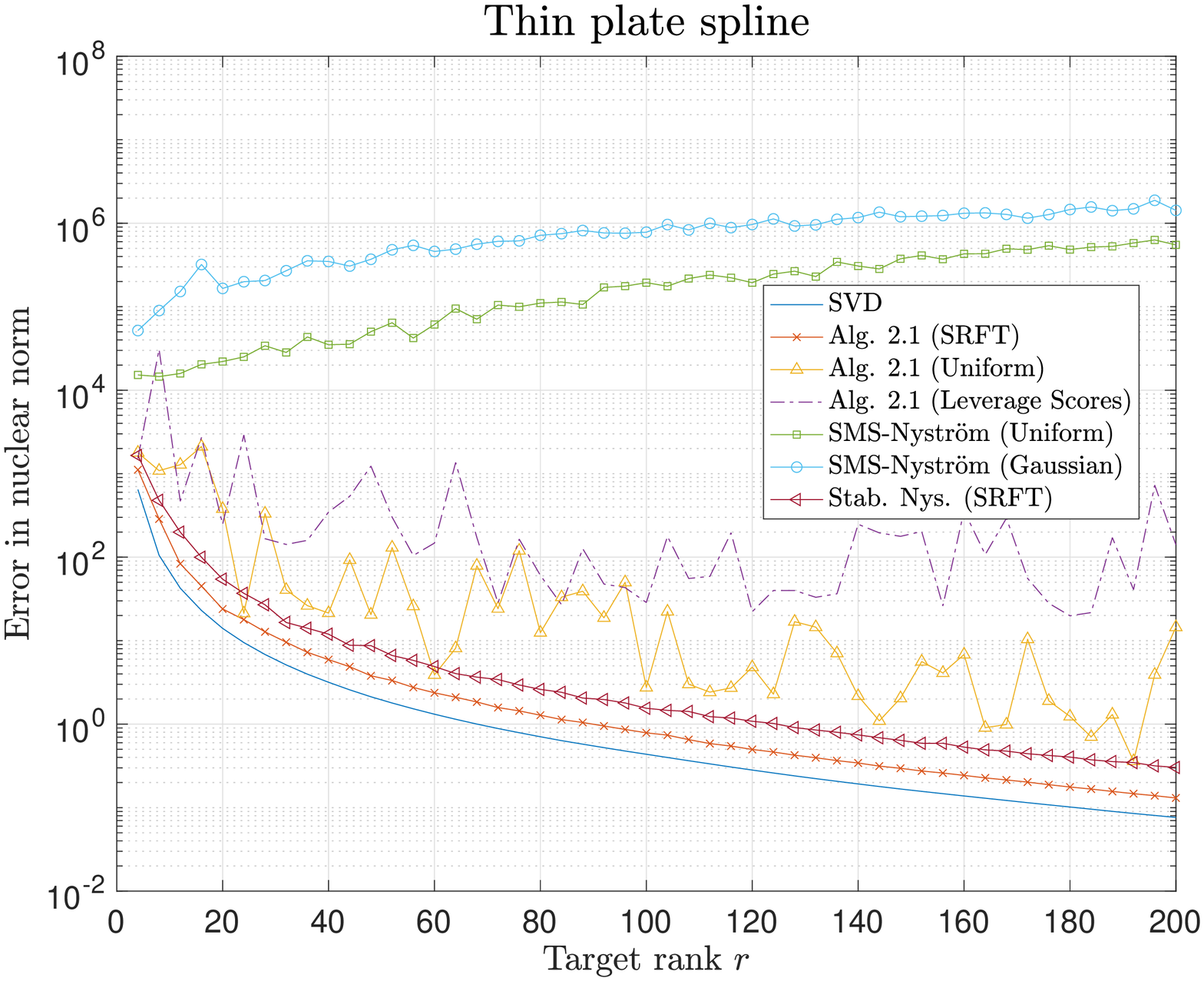}} 
\centering 
\caption{Comparison of different methods for symmetric indefinite matrices: SMS-Nystr\"om \cite{smsnystrom}, stabilized Nystr\"om \cite{indefnys} and Algorithm \ref{alg:nys}. The first two methods and Algorithm \ref{alg:nys} using uniform column sampling and leverage score column sampling can fail on some kernels while Algorithm \ref{alg:nys} using the SRFT sketch (random embedding) works well for all the kernels in the experiment.}
\label{indefkerfig}
\end{figure}

\subsection{Dataset examples} \label{subsec:dataset}
We now compare the three different methods using two different high-dimensional datasets, the Covertype and the Anuran Calls (MFCC) from the UC Irvine Machine Learning Repository \cite{UCIdata}. We illustrate the following algorithms
\begin{enumerate}
    \item Algorithm \ref{alg:nys} with the SRFT sketch and the sketch size $s = 2r$,
    \item Algorithm \ref{alg:nys} with k-means++ samples and the sketch size $s = 2r$,
    \item Algorithm \ref{alg:nys} with uniform column sampling and the sketch size $s = 2r$,
    \item Stabilized Nystr\"om with k-means++ samples, the sketch size $s = r$ and $\epsilon = 10^{-14}$
\end{enumerate} where $r$ is the target rank. We use the following kernel functions
\begin{enumerate}
    \item Thin plate spline kernel: $\norm{x-y}^2 \log\left(\norm{x-y}^2\right)$
    \item Sigmoid kernel: $\text{tanh}\left(1+\norm{x-y}^2\right)$
    \item Multiquadric kernel: $\sqrt{1+\norm{x-y}^2}$
\end{enumerate} with the datasets
\begin{enumerate}
    \item Covertype $(n = 581012)$ with dimension $d = 54$,
    \item Anuran Calls (MFCC) $(n = 7195)$ with dimension $d=22$.
\end{enumerate} For each dataset, we sample $n = 4000$ data uniformly at random and then center the mean and normalize all features to have variance $1$.

The results are illustrated in Figure \ref{fig:UCI}. We observe that the cause of instability in the Nystr\"om approximation for symmetric indefinite matrices is not necessarily coming from the core matrix $W$ having very small singular values as Stabilized Nystr\"om can give unstable approximations as seen in Figure \ref{fig:UCI}. Also, although Algorithm \ref{alg:nys} using k-means++ samples is more accurate than uniform column sampling, they both do not give robust low-rank approximations. This shows that it is difficult to find a column sampling scheme that guarantees stable Nystr\"om approximation for symmetric indefinite matrices. On the other hand, Algorithm \ref{alg:nys} using the SRFT sketch gives robust approximation throughout the experiment and sometimes outperforms the other methods in this experiment such as in Figure \ref{aa} and \ref{ee}.

\begin{figure}[!ht]
\hspace*{-0.5cm}
\subfloat[]{\label{aa}\includegraphics[scale = 0.35]{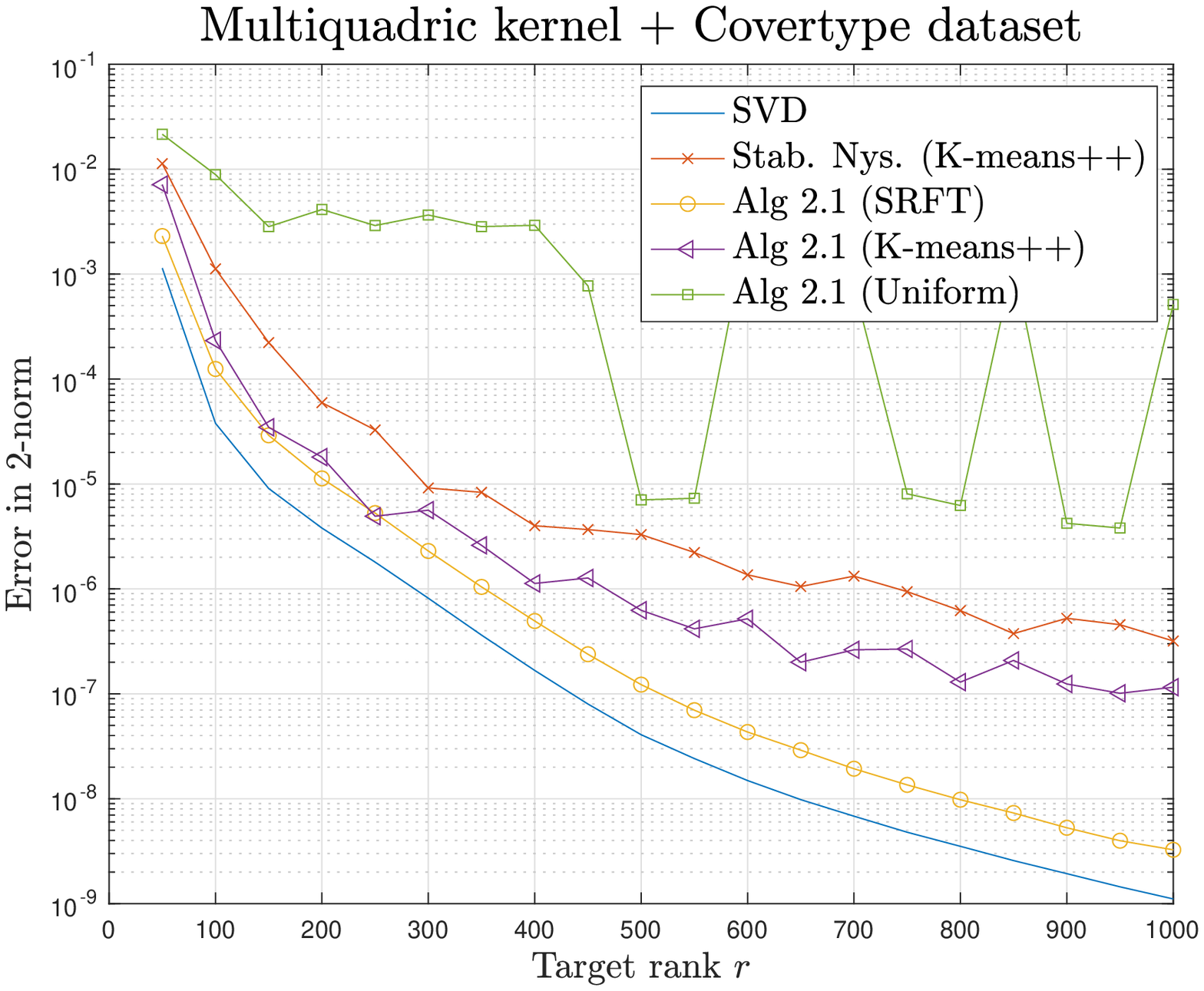}}
\hspace*{-0.5cm}
\subfloat[]{\label{bb}\includegraphics[scale = 0.35]{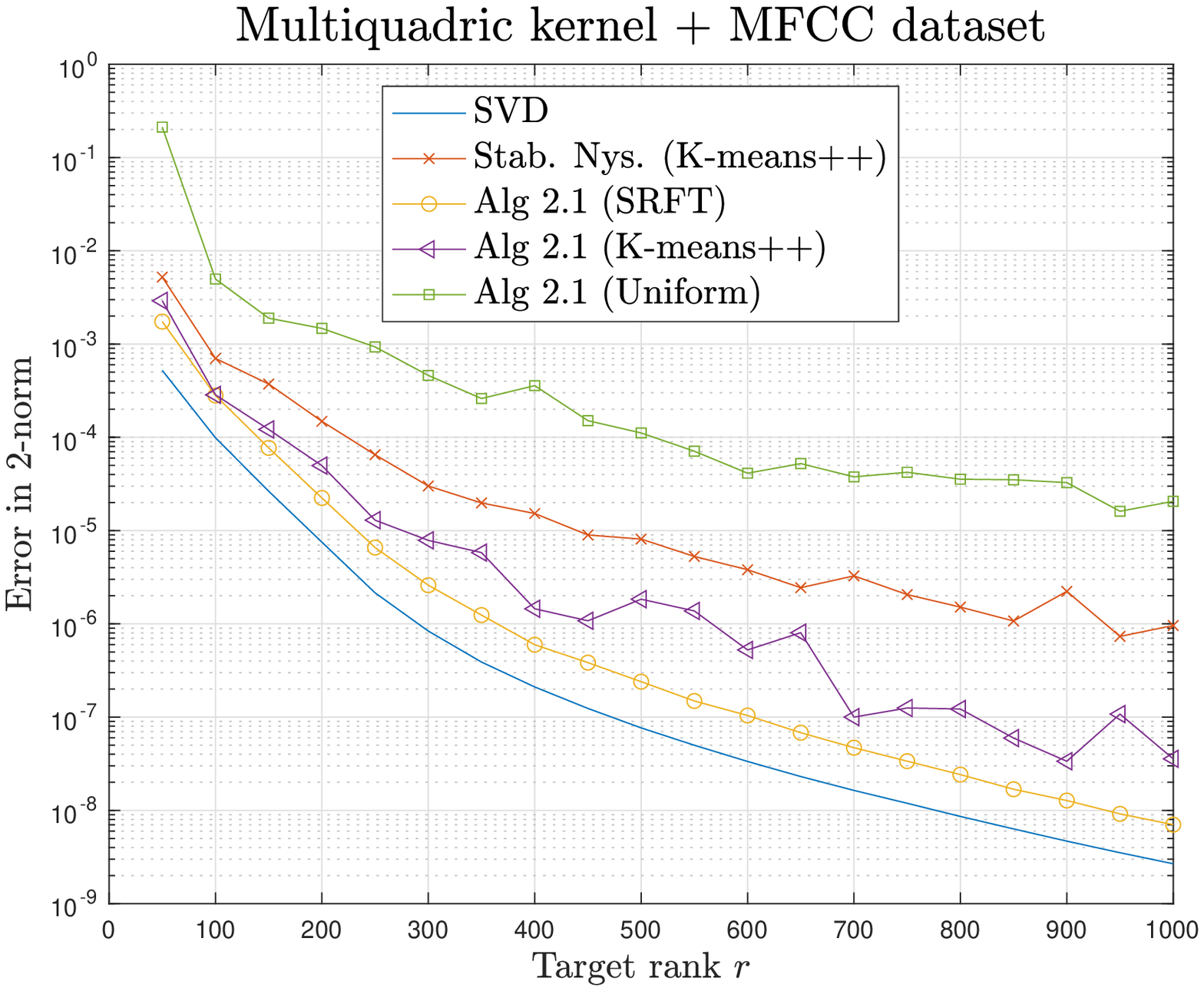}} \\
\hspace*{-0.5cm}
\subfloat[]{\label{cc}\includegraphics[scale = 0.35]{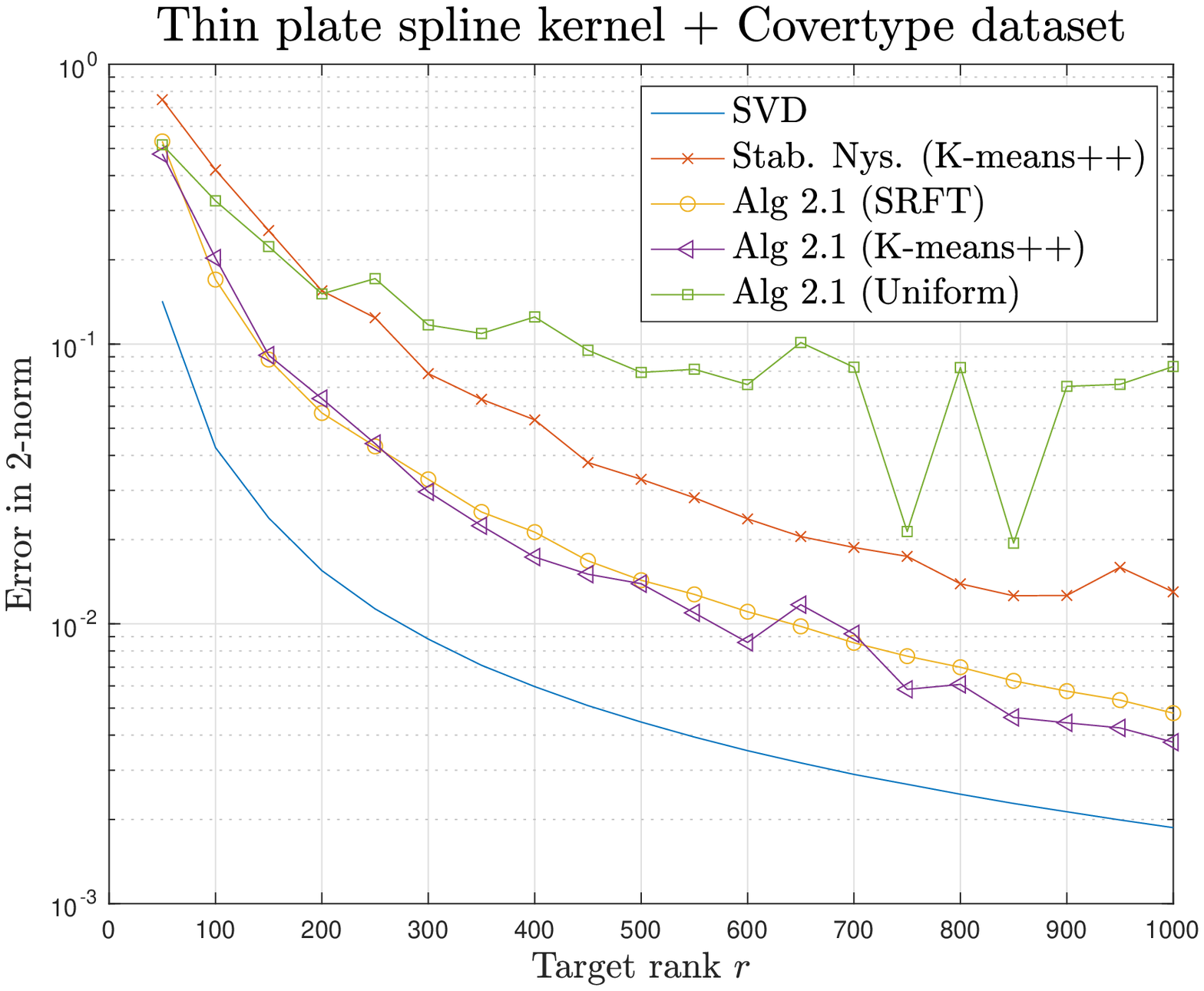}}
\hspace*{-0.5cm}
\subfloat[]{\label{dd}\includegraphics[scale = 0.35]{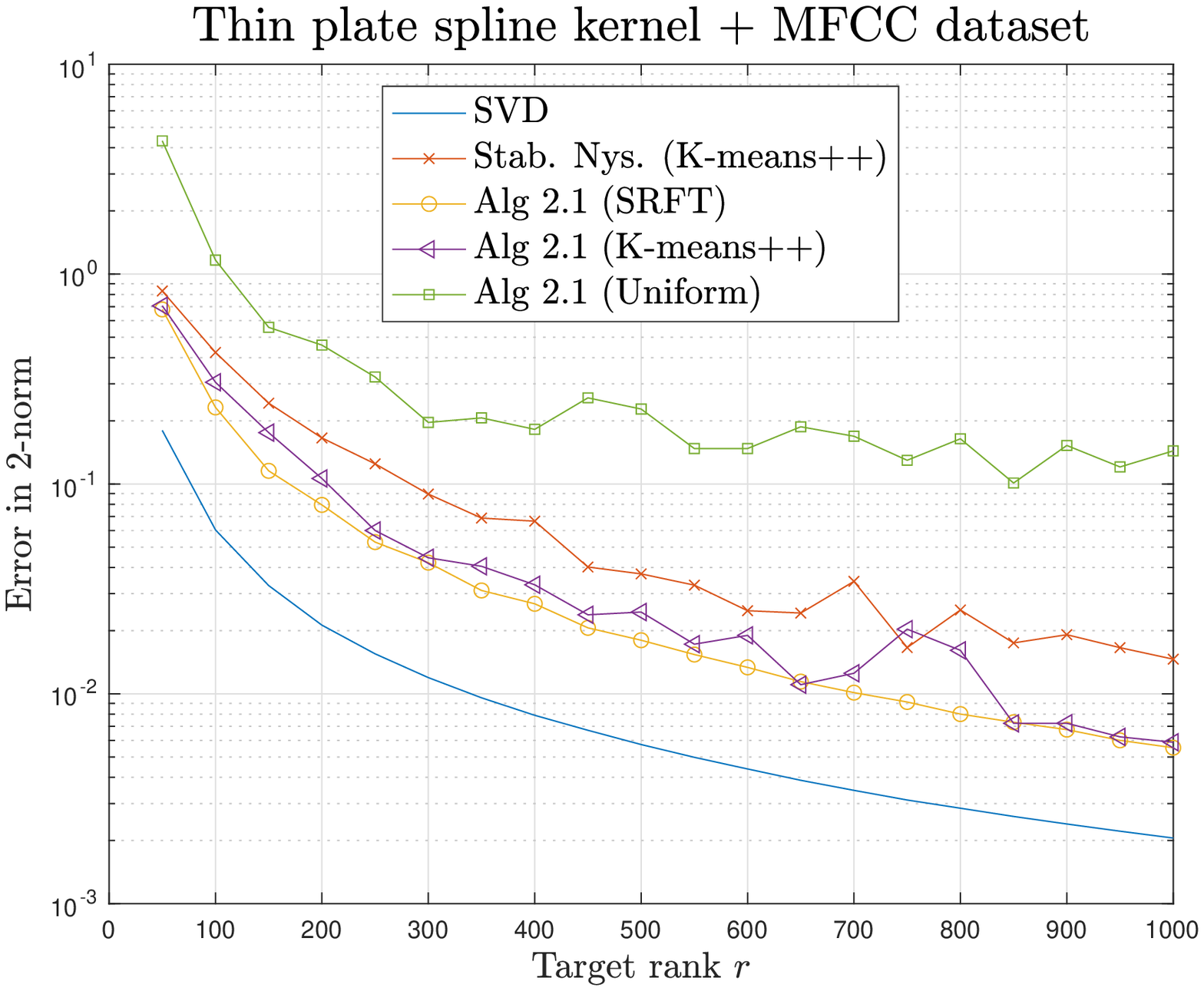}} \\
\hspace*{-0.5cm}
\subfloat[]{\label{ee}\includegraphics[scale = 0.35]{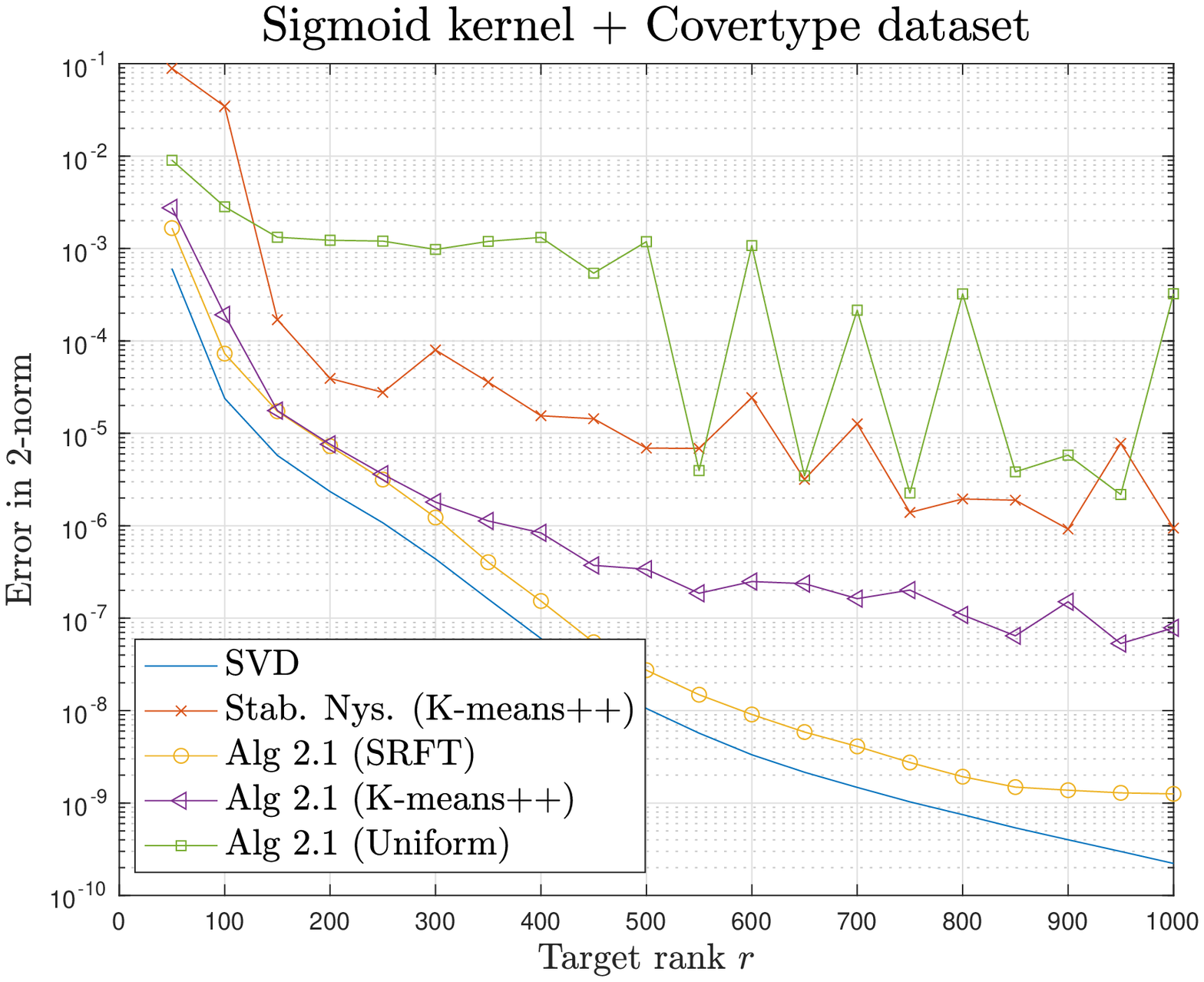}}
\hspace*{-0.5cm}
\subfloat[]{\label{ff}\includegraphics[scale = 0.35]{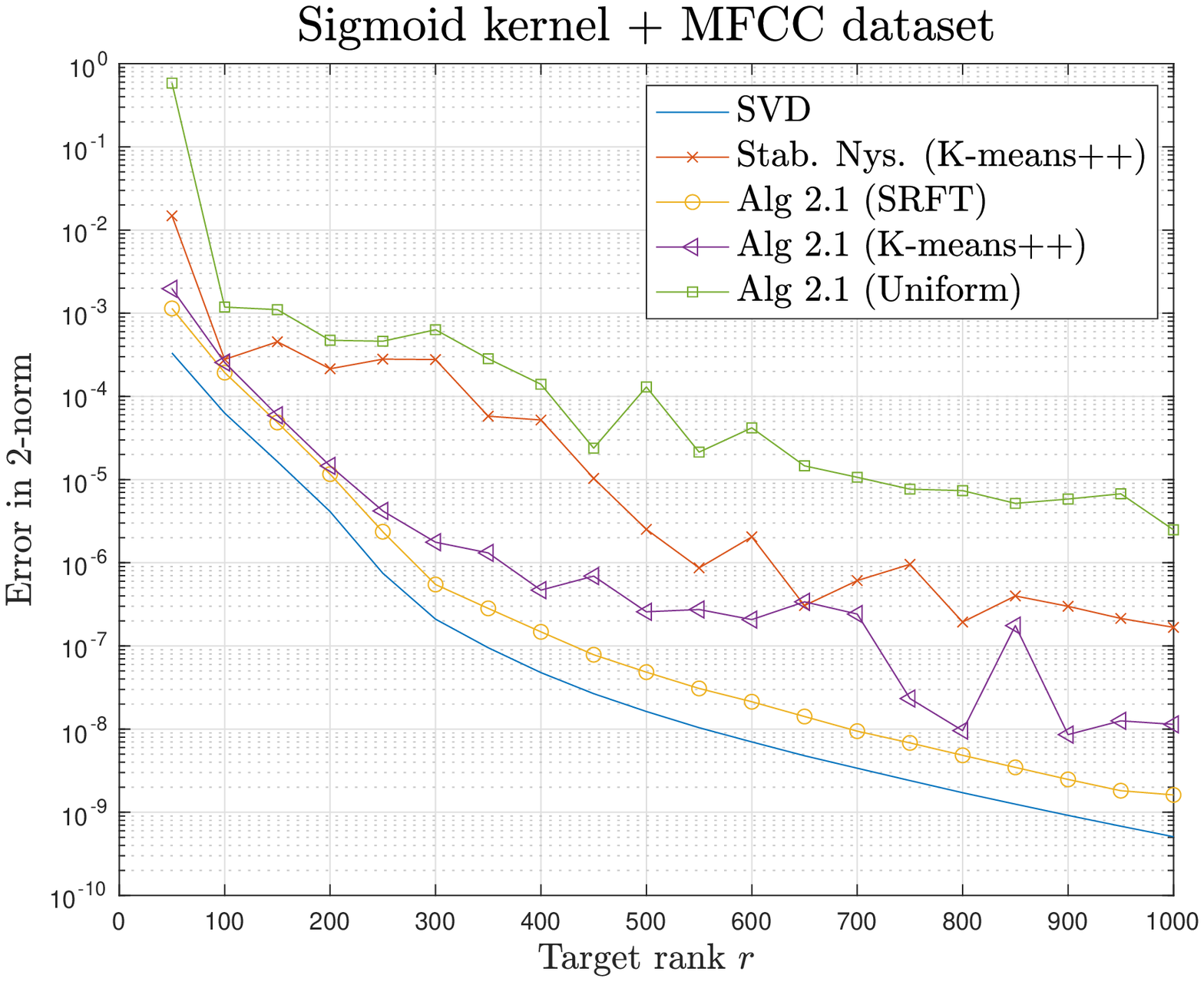}} 
\centering 
\caption{Comparison of stabilized Nystr\"om \cite{indefnys} and Algorithm \ref{alg:nys} for symmetric indefinite matrices using three different indefinite kernels and two different datasets. Stabilized Nystr\"om method and Algorithm \ref{alg:nys} using k-means++ samples and uniform column sampling can give unstable low-rank approximation while Algorithm \ref{alg:nys} using the SRFT sketch (random embedding) gives robust approximation throughout the experiment.}
\label{fig:UCI}
\end{figure}

\section{Discussion} \label{discussion}
Much of the literature on approximating symmetric matrices using any of the variants of the Nystr\"om method is based on column sampling. In this work, we used random embeddings for our algorithm (Algorithm \ref{alg:nys}) and a special class of random embeddings for the analysis, namely Gaussian embeddings. Random embeddings were used as they are more robust than column sampling, and Gaussian embeddings were used for analysis because we can leverage their rich theoretical properties. The general behaviour when we use the Nystr\"om method with column sampling matrices on symmetric indefinite matrices is unknown. In Figure \ref{indefkerfig}, we see that the two frequently used column sampling schemes, uniform sampling and leverage score sampling can be unstable. It appears to be difficult to find a column sampling scheme that guarantees robust Nystr\"om approximation for symmetric indefinite matrices and, to our knowledge, is an open problem. We hope that our results would shed light on the development of a robust indefinite Nystr\"om method based on column subsampling.

\subsection*{Acknowledgements} We thank the anonymous referees and the editor for their many insightful comments and suggestions, which helped us to improve the quality of the paper.


\bibliographystyle{siamplain}
\bibliography{references}

\end{document}